\documentclass[12pt,a4paper]{amsart}   

\usepackage[T1]{fontenc}
\usepackage[cp1252]{inputenc}            
\usepackage[dvips]{graphicx}
\usepackage{amssymb}

\newtheorem{lause}{Theorem}               
\newtheorem{maar}[lause]{Definition}    

{\par\vspace{2.5mm}\noindent\textit{Proof.\ }}
{\qed\par\vspace{2.5mm}}

\numberwithin{equation}{section}
\numberwithin{lause}{section}

\allowdisplaybreaks
\begin{document}
\thispagestyle{empty}



\title{Periodic orbits $1-5$ of quadratic polynomials on a new coordinate plane}

\author{Pekka Kosunen \textsuperscript{1}}

\date{\today}



\maketitle

\noindent

\address{Department of Physics and Mathematics, University of Eastern Finland, P. O. Box 111, FI-80101 Joensuu, Finland;}
\email{pekka.kosunen@uef.fi}
\footnotetext[1]{The author wishes to acknowledge partial support from the Academy of Finland grant 286877.}

\begin{abstract}
While iterating the quadratic polynomial $f_{c}(x)=x^{2}+c$ the degree of the iterates grows very rapidly, and therefore solving the equations corresponding to periodic orbits becomes very difficult even for periodic orbits with a low period. In this work we present a new iteration model by introducing a change of variables into an $(u,v)$-plane, which changes the situation drastically. The parametrization in the new model is simpler in the sense that, in it, the equations of the periodic orbits are of lower degree than the ones in previous models. As an excellent example of this we can compare equations of orbits period four on  $(x,c)$- and $(u,v)$-planes. In the latter case, this equation is of degree two with respect to $u$ and it can be solved explicitly. In former case the corresponding equation $((((x^{2}+c)^{2}+c)^{2}+c)^{2}+c-x)/((x^{2}+c)^{2}+c-x)=0$ is of degree $12$ and it is thus much more difficult to solve.
\end{abstract}

\section{Introduction}

The dynamics of quadratic polynomials is often studied by using the family of maps  $f_{c}(x)=x^{2}+c$, where $c\in \mathbb{C}$ (see, for example, \cite{4, 14, 25, 3}). The orbit of the point $x_{0}\in \mathbb{C}$ is the sequence of points
$x_{0},x_{1},x_{2},\ldots$, where $x_{n}=f_{c}(x_{n-1})=f_{c}^{n}(x_{0})$.
In this work a central role is played by the periodic orbits or cycles of $f_{c}$ for which $f_{c}^{n}(x_{0})=x_{0}$ for some $n\in\mathbb{N}$. Now the number $n$ is the period of the orbit and $x_{0},x_{1},x_{2},\ldots x_{n-1}$ are periodic points of period $n$.

In this paper, we obtain the equations for the periodic orbits $1-5$ of the family $f_{c}$ by iterating the function
\begin{equation}
 G(u,v)=\left(\frac{-u+v+uv}{u},\frac{u^{2}-u+v-u^{2}v-uv+uv^{2}+v^{2}}{u}\right), \label{1}
\end{equation}
and forming the corresponding iterating system on the $(u,v)$-plane (in Section $2$). The function $G$ is a two-dimensional quadratic polynomial map which is defined in the complex 2-space $\mathbb{C}^{2}$ and its iteration reveals the dynamics of $f_{c}$ in a way explained below. We will also compare the $(u,v)$-plane model to the $(x,y)$-plane model, which was introduced by Erkama in \cite{1}.
The purpose of developing this new model is that we obtain equations of periodic orbits of lower degree, which are easier to handle. The main novelty of this article is, as in the article of Erkama \cite{1}, the combination of parameter space ($c \in \mathbb{C}$), (e.g., a discussion of the Mandelbrot set) and dynamical space ($x \in \mathbb{C}$), (e.g., a discussion of periodic points and the filled Julia set). These papers combine to develop this method of modeling quadratic polynomial dynamics.

 In the article \cite{1} Erkama studied the dynamics of the polynomial family $f_{c}$ by iterating the function
\begin{displaymath}
 F(x,y)=(y,y^{2}+y-x^{2})
 \end{displaymath}
on the $(x,y)$-plane, thus obtaining
\begin{eqnarray*}
&&P_{0}(x,y)=x_{0}=x , \\
&&P_{n+1}(x,y)=x_{n+1}=P_{n}(x,y)^{2}+y-x^{2}, \qquad n=1,2,3,\ldots.
\end{eqnarray*}
 In this case $(x,y)$ is fixed by $F^{n}$, so $F^{n}(x,y)=(x,y)$, if and only if $P_{n}(x,y)=x$. The set of such points is the union of all orbits whose period divides $n$. Because $F$ leaves each parabola $y=f_{c}(x)$ invariant, the dynamics of $f_{c}$ can be studied by studying the restriction of $F$ to this parabola. Now $F^{n}(x,f_{c}(x))=(f_{c}^{n}(x),f_{c}(f_{c}^{n}(x))$, so $x$ is a periodic point of $f_{c}$ if and only if $(x,f_{c}(x))$ is a periodic point of $F$ with the same period.
The formula $P_{n}=x_{n}$ was given by Erkama as
\begin{equation}
x_{n}=x_{1}+(x_{1}-x_{0})\sum_{\nu=1}^{n-1}(x_{0}+x_{n-1})(x_{0}+x_{n-2})\ldots (x_{0}+x_{n-\nu}),\label{3}
\end{equation}
where $n \geq 2$.

 Periodic orbits form algebraic curves corresponding to their periods. In this paper we present (in Sections $3-6$) the equations of these curves for periods $1-5$ on the $(u,v)$-plane. We also compare the figures of the real curves on the $(u,v)$-plane as well as on the $(x,y)$-plane. Figures and period equations of the $(x,y)$-plane for periods $1-4$ have appeared in \cite{1}, but period five in both cases and all models of the $(u,v)$-plane are previously unpublished. Moreover, we study the locations of neutral, attracting and repulsive fixed points (see \cite{11, 16}) on these curves  by using a suitable eigenvalue formula.
By plotting attracting periodic points of the function~$G$ we obtain a new version of the well-known period doubling bifurcation diagram (see, for example, \cite{6, 17, 8}) where the algebraic curves are of lower degree than previously (see Figure $\ref{kuva19}$). In the case of the function $F(x,y)$ we obtain Figure $\ref{kuva18}$.

Applications of the $(u,v)$-plane model for quadratic dynamical systems have been published in \cite{27} and \cite{37}.

\section{A new two-dimensional model for the quadratic family $f_{c}$ on the $(u,v)$-plane}

It is well known that the dynamics of each quadratic polynomial is equivalent to the dynamics of $f_{c}$ for some $c \in \mathbb{C}$. In this section we form a new iteration model in which equations of periodic orbits are of remarkably lower degree than in the earlier models. The new model is obtained from the $(x,y)$-plane model in \cite{1} by the change of parameters
\begin{eqnarray*}
u&=&x+y, \\
v&=&x+y^2+y-x^2.
\end{eqnarray*}
\clearpage
\begin{figure}[h!]
\begin{center}
\includegraphics[width=0.58\textwidth,angle=270]{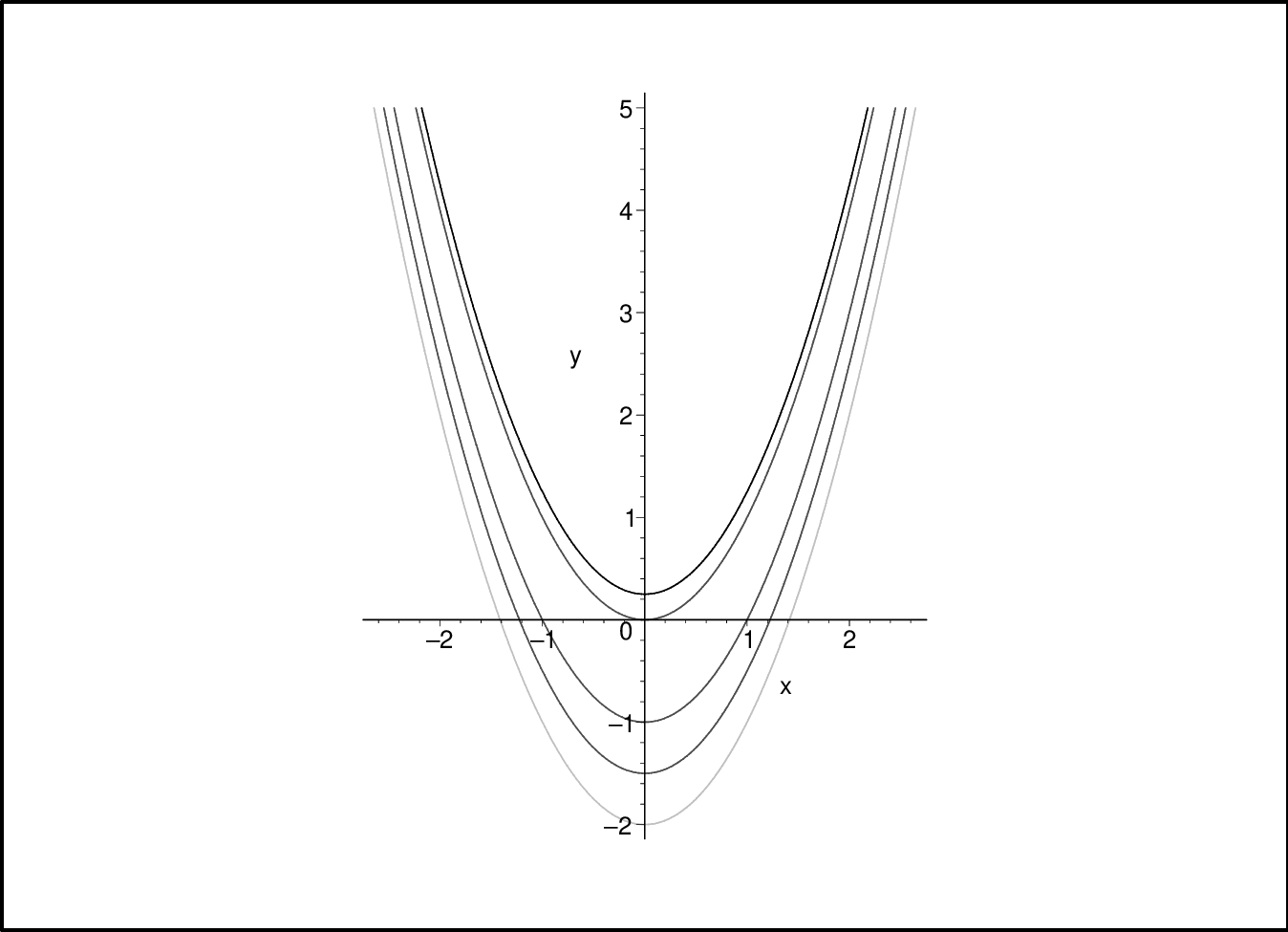}
\caption{Curves $c=y-x^{2}$ on the $(x,y)$-plane with the parameter values $c=1/4$, $c=0$, $c=-1$, $c=-3/2$ and $c=-2$.} \label{kuva2}
\end{center}
\end{figure}
\begin{figure}[h!]
\begin{center}
\includegraphics[width=0.58\textwidth,angle=270]{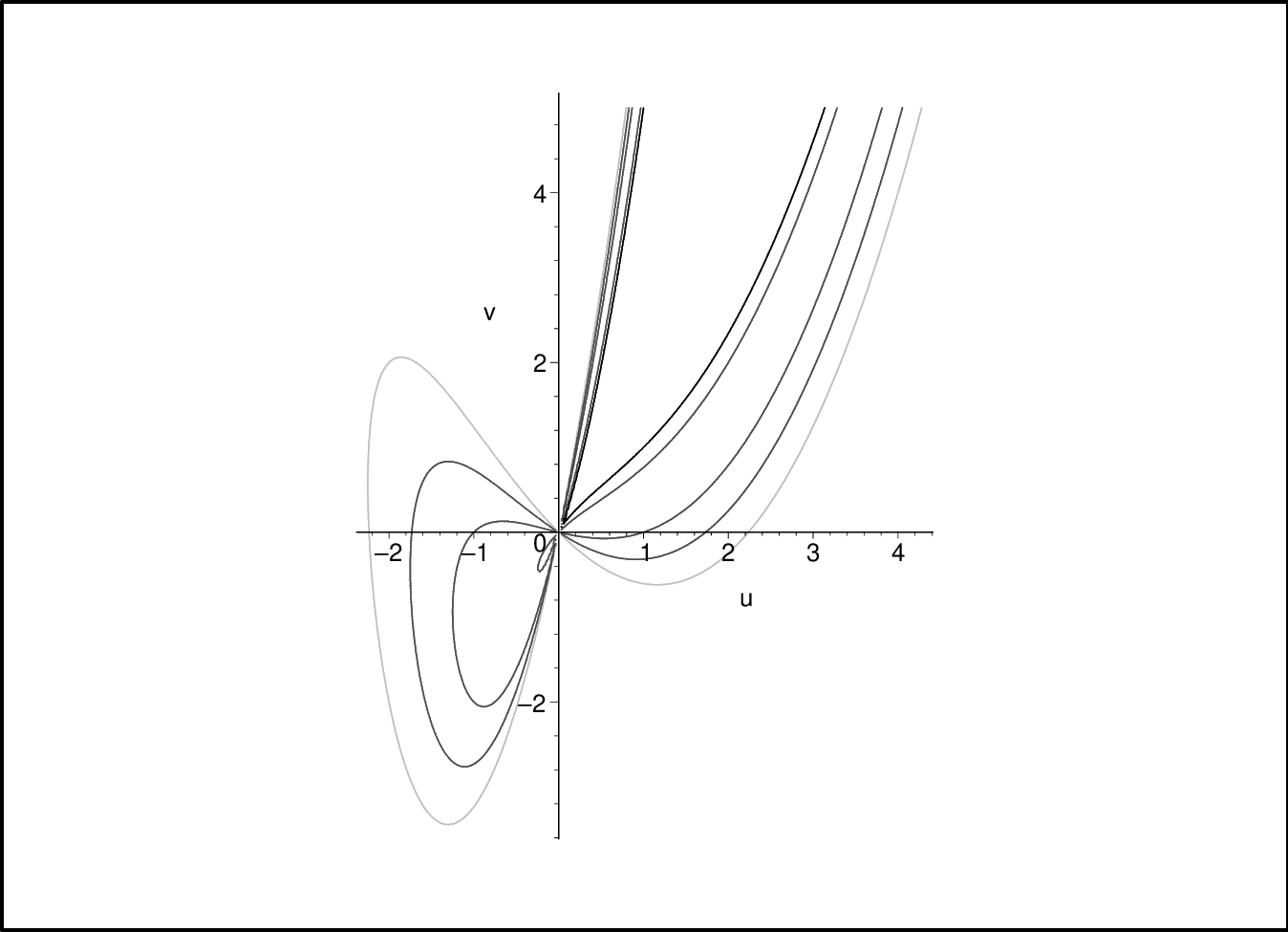}
\caption{Curves $c=(-u^{4}+u^{2}(2v-3)-v^{2}+4uv)/(4u^{2})$ on the $(u,v)$-plane with the parameter values $c=1/4$, $c=0$, $c=-1$, $c=-3/2$ and $c=-2$.} \label{kuva3}
\end{center}
\end{figure}
\clearpage
Note that already this model obtained by Erkama on the $(x,y)$-plane is of lower degree than other earlier models. Let us discuss first some essential properties of the $(u,v)$-plane.
\begin{lause} \label{lause1.2}
The point of indeterminacy (singularity) of the $(x,y)$-plane correspond to point $(u,v)=(0,0)$ on the $(u,v)$-plane and moreover poles correspond to points $(u,v)=(\infty,v)$, $v \in \mathbb{C}$ and $(u,v)=(0,v)$, $v \neq 0$.
\end{lause}
\begin{proof}
According to
($\ref{3}$) we get
\begin{displaymath}
x_{0}+x_{2}=(x_{0}+x_{1})[1+(x_{1}-x_{0})]=(x_{0}+x_{1})[1+(x_{1}+x_{0})-2x_{0}].
\end{displaymath}
Since $u=x_{0}+x_{1}$ and $v=x_{0}+x_{2}$, it follows by the previous formula that
\begin{displaymath}
v=u[1+u-2x_{0}]=u+u^{2}-2ux.
\end{displaymath}
Now $y=u-x$ so we can solve $x$ and $y$ in terms of variables $u$
and $v$ as follows
\begin{equation}
 x=\frac{u^{2}+u-v}{2u},\qquad y=\frac{u^{2}-u+v}{2u}.\label{4}
 \end{equation}
 We see easily that poles of both of these expressions appear exactly when $u=\infty$, $v \in \mathbb{C}$ or $u=0$, $v \neq 0$. Furthermore, $(u,v)=(0,0)$ is a point of indeterminacy.
\end{proof}
In the $(x,y)$-plane the mapping
\begin{equation}
c:(x,y)\mapsto y-x^{2} \label{4.0}
\end{equation}
projects each parabola $y=f_{c}(x)$ to the corresponding point $c$ in the parameter space. In that case the period-$\displaystyle{n}$ curves are tangent to parabolas $y=f_{c}(x)$ at the intersection points where period doubling occurs (see \cite{1}, Figure $1$).
The same way on the $(u,v)$-plane using formulas ($\ref{4}$), the mapping
 \begin{equation}
c:(u,v)\mapsto \frac{-u^{4}+u^{2}(2v-3)-v^{2}+4uv}{4u^{2}} \label{4.1}
\end{equation}
projects a curve to the corresponding point $c$ in the parameter space. Similarly on the $(u,v)$-plane the period-$\displaystyle{n}$ curves are tangents to level curves of ($\ref{4.1}$) at the bifurcation points. Curves of $c(x,y)=y-x^{2}$ and $c(u,v)=(-u^{4}+u^{2}(2v-3)-v^{2}+4uv)/(4u^{2})$ are presented (Figures $\ref{kuva2}$ and $\ref{kuva3}$) with parameter values $c=1/4$, $c=0$, $c=-1$, $c=-3/2$ and $c=-2$.

The classical Mandelbrot set consists of such $c \in \mathbb{C}$ for which the sequence $\{G^{n}(c,c^{2}+c)\}$ is bounded.
 The sequence $\{G^{n}(c,c^{2}+c)\}$ is bounded in the space $\mathbb{C}^{2}$, if and only if the sequence $\{f_{c}^{n}(0)\}$ is bounded in the space~$\mathbb{C}$.
 (In the same way, the sequence $\{F^{n}(0,c)=(f_{c}^{n}(0),f_{c}(f_{c}^{n}(0))\}$ is bounded in the space $\mathbb{C}^{2}$, if and only if the sequence $\{f_{c}^{n}(0)\}$ is bounded in the space~$\mathbb{C}$.)
It is well known that the intersection of the Mandelbrot set of $f_{c}(x)$ with $\mathbb{R}$ is $[-2, 1/4]$.

Since $f_{c}'(x)=0$ only when $x=0$, the line $x=0$ includes one critical point at every critical cycle of every period. For this reason we call the line $x=0$ the $\displaystyle{critical \ curve}$ of the function $F$. The so called $\displaystyle{central \ critical \ point}$ is a critical point of the critical cycle, which is located on the critical curve. Due to this, the central critical points of $F$ lie on the line $x=0$.
Correspondingly, the curve $v=u^{2}+u$ (see Figure $\ref{kuva140}$) includes all central critical points on the $(u,v)$-plane when
\begin{eqnarray}
u&=&x_{0}+x_{1}=x+y, \label{1abc} \\
v&=&x_{0}+x_{2}=x+y^2+y-x^2=u(1+y-x). \label{1bcd}
\end{eqnarray}
So by calculating central critical points first, we obtain other points of the cycle by iterating, like we see in Sections 4, 5 and 6 below.

Let us define the universal filled Julia set  $K(G)$ as the set of all points  $(u,v) \in \mathbb{C}^{2}$ such that the sequence $\{G^{n}(u,v)\}$ is bounded.
Then the Mandelbrot set can also be described as the intersection of $K(G)$ with the parabola $v=u^2+u$. Thus the dynamical plane and the parameter space are simultaneously present also in this model.

In our new $(u,v)$-plane model we write
\begin{eqnarray*}
(u_{0},v_{0}) &=& (x_{0}+x_{1},x_{0}+x_{2})=(x+y,x+y^{2}+y-x^{2}),\\
(u_{1},v_{1}) &=& (x_{1}+x_{2},x_{1}+x_{3}) \\
 &=& (y+y^{2}+y-x^{2},y+(y^{2}+y-x^{2})^{2}+y-x^{2}).
\end{eqnarray*}
By using formulas ($\ref{4}$) we get
\begin{displaymath}
u_{1}=\frac{-u+v+uv}{u}
\end{displaymath}
and
\begin{displaymath}
v_{1}=\frac{u^{2}-u+v-u^{2}v-uv+uv^{2}+v^{2}}{u}.
\end{displaymath}
Now the dynamics of the $(u,v)$-plane is determined be the iteration of the function $G$ defined by
\begin{eqnarray*}
 G(u,v)&=&\left(\frac{-u+v+uv}{u},\frac{u^2-u+v-u^2v-uv+uv^2+v^2}{u}\right) \\
 &=&(v+v/u-1,(v+v/u-1)(1+v-u)) \\
 &=&(g_{1},g_{1}(1+v-u)),
\end{eqnarray*}
where
\begin{displaymath}
g_{1}=\frac{-u+v+uv}{u}.
\end{displaymath}
\begin{maar}
If
\begin{eqnarray*}
G(u,v)&=&(R(u,v),Q(u,v)) \\
&=&\left(\frac{-u+v+uv}{u},\frac{u^2-u+v-u^2v-uv+uv^2+v^2}{u}\right),
\end{eqnarray*}
then recursively
\begin{equation}
\left\{%
\begin{array}{l}
      (R_{0}(u,v),Q_{0}(u,v))=(u,v)  \\ \\
      (R_{n+1}(u,v),Q_{n+1}(u,v))=G(R_{n}(u,v),Q_{n}(u,v)),\label{5}
\end{array}%
\right.
\end{equation}
where
\begin{equation}
\left\{%
\begin{array}{l}
      R_{n+1}(u,v)=Q_{n}(u,v)-1+Q_{n}(u,v)/R_{n}(u,v) \\ \\
      Q_{n+1}(u,v)=R_{n+1}(u,v)(1+Q_{n}(u,v)-R_{n}(u,v)),\label{6}
\end{array}%
\right.
\end{equation}
when $n=0,1,2,\ldots$.
\end{maar}
 Now $(u,v)$ is fixed by $G^{n}$, so $G^{n}(u,v)=(u,v)$, if and only if $(R_{n}(u,v),Q_{n}(u,v))$=$(u,v)$. The set of such points is the union of all orbits whose period divides $n$. Since $G$ leaves each level curve of $c(u,v)=(-u^{4}+u^{2}(2v-3)-v^{2}+4uv)/(4u^{2})$ invariant, the dynamics of $f_{c}$ can be studied now by studying the restriction of $G$ to this curve.

Due to the form of the definition of an eigenvalue on the $(u,v)$-plane, it is convenient to denote
$T_{n+1}(u,v)=1+Q_{n}(u,v)-R_{n}(u,v)$, and write the formula $(\ref{6})$ in the form
\begin{displaymath}
\left\{%
\begin{array}{l}
      R_{n+1}(u,v)=Q_{n}(u,v)-1+Q_{n}(u,v)/R_{n}(u,v) \\ \\
      Q_{n+1}(u,v)=R_{n+1}(u,v)(1+Q_{n}(u,v)-R_{n}(u,v))=R_{n+1}(u,v)T_{n+1}(u,v).\label{622}
\end{array}%
\right.
\end{displaymath} \\
Now the Jacobian of the function $G^{n}(u,v)=G(R_{n-1}(u,v),Q_{n-1}(u,v))$ \\
$=(R_{n}(u,v),Q_{n}(u,v))=(R_{n}(u,v),R_{n}(u,v)T_{n}(u,v))$ is of the form

\begin{displaymath}
\begin{array}{ll}
\displaystyle{\det(J_{G})}&=\displaystyle{\det} \displaystyle{\begin{pmatrix}
  \displaystyle\frac{\partial R_{n}}{\partial u} &  \displaystyle\frac{\partial R_{n}}{\partial v} \\
  \displaystyle\frac{\partial Q_{n}}{\partial u} &  \displaystyle\frac{\partial Q_{n}}{\partial
  v}
\end{pmatrix}}=\displaystyle{\det} \displaystyle{\begin{pmatrix}
  \displaystyle\frac{\partial R_{n}}{\partial u} &  \displaystyle\frac{\partial R_{n}}{\partial v} \\
  \displaystyle\frac{\partial R_{n}}{\partial u}T_{n}+R_{n}\frac{\partial T_{n}}{\partial
  u}   & \displaystyle\frac{\partial R_{n}}{\partial v}T_{n}+R_{n}\frac{\partial T_{n}}{\partial
  v}
\end{pmatrix}}\\ \\
&=\displaystyle{R_{n}\Big(\frac{\partial R_{n}}{\partial u}\frac{\partial T_{n}}{\partial
  v}-\frac{\partial R_{n}}{\partial v}\frac{\partial T_{n}}{\partial  u}\Big)}.
  \end{array}
\end{displaymath}

For the periodic points $(u,v)$ of period $n$ we have $R_{n}(u,v)=u$. Because of this, the eigenvalue $\lambda_{n}(u,v)$ of a periodic point of period $n$ of the function $G^{n}(u,v)$ is of the form
\begin{equation}
    \lambda_{n}(u,v)=u\Big(\frac{\partial R_{n}}{\partial u}\frac{\partial T_{n}}{\partial
  v}-\frac{\partial R_{n}}{\partial v}\frac{\partial T_{n}}{\partial  u}\Big),  \label{50}
\end{equation} \\
when $n=1,2,\ldots$.

On the $(u,v)$-plane fixed points and periodic points are classified in the following way:
\begin{maar} \label{luokitus}
Let us assume that $(u,v) \in \mathbb{C}^{2}$ is a periodic point of period $n$ of the function $G(u,v)$. In that case
$(u,v)$ is
\begin{eqnarray*}
1.&&  \textrm{attracting}, \ \textrm{if} \ 0<|\lambda_{n}(u,v)|<1, \\ \nonumber
2.&&  \textrm{super-attracting or critical}, \ \textrm{if} \ \lambda_{n}(u,v)=0, \\ \nonumber
3.&&  \textrm{repulsive}, \ \textrm{if} \ |\lambda_{n}(u,v)|>1, \\ \nonumber
4.&&  \textrm{indifferent (neutral)}, \ \textrm{if} \ |\lambda_{n}(u,v)|=1, \nonumber
\end{eqnarray*}
when $n=1,2,3, \ldots$.
 \end{maar}

\section{Periodic orbits of periods $1-2$}

In this section we obtain periodic orbit equations of period one and two, and classify these points based on the Definition $\ref{luokitus}$ on the $(u,v)$-plane.

Fixed points of period one on the $(u,v)$-plane satisfy a pair of equations \\ $G(R_{0}(u,v),Q_{0}(u,v))=(u,v)$, where
\begin{equation}
\left\{%
\begin{array}{l}
      R_{1}(u,v)=(-u+v+uv)/u=u \\
      Q_{1}(u,v)= -(-u^{2}+u-v+u^{2}v+uv-uv^{2}-v^{2})/u=v.\label{61}
\end{array}%
\right.
\end{equation}
The first equation is equivalent with
\begin{displaymath}
u^{2}+u-v-uv=(u-v)(1+u)=0.
\end{displaymath}
When $u=-1$, the second equation is satisfied only when $v=-1$, so $v=u$ is the only common solution for both of them. Thus $v=u$ is the equation of period one orbits.
We obtain neutral periodic points of period one when $v=u$, and by ($\ref{50}$) the eigenvalue is
\begin{displaymath}
|\lambda_{1}(u,v)|=\left|\frac {{u}^{2}+u-v}{u}\right|=1,
\end{displaymath}
 and thus the neutral periodic points are $(u,v)=(-1,-1)$ and $(u,v)=(1,1)$ (corresponding on the $(x,y)$-plane to $(x,y)=(-1/2,-1/2)$ and $(x,y)=(1/2,1/2)$).
Thus the only real bifurcation point is $(u,v)=(-1,-1)$ (corresponding on the $(x,y)$-plane to $(x,y)=(-1/2,-1/2)$), which we obtained as the intersection of period one and two orbits (Figures $\ref{kuva9}$ and $\ref{kuva9.1}$). The point $(u,v)=(1,1)$ corresponds to $c=1/4$, the cusp of the main cardioid of the classical Mandelbrot set. The point $(u,v)=(-1,-1)$ corresponds to $c=-3/4$.
The origin is the only critical point (or super-attracting point) and naturally it is located on the parabola $v=u^{2}+u$. Because of that the line segment
\begin{displaymath}
\{(u,v)|v=u,u\in[-1,1]\}
\end{displaymath}
is located in the Fatou set $F_{G}$. We see this clearly in the bifurcation diagram (Figure $\ref{kuva19}$).
Similarly the origin is the only critical point on the $(x,y)$-plane.
 Naturally the complement
 \begin{displaymath}
V_{1} \setminus \{(u,v)|v=u,u\in[-1,1]\},
\end{displaymath}
where $V_{1}=\{(u,v)|v-u=0\}$, is located in the Julia set~$J_{G}$.

As for periodic points of period two, they satisfy the pair of equations $G^{2}=(u,v)$, or
\begin{equation}
\left\{%
\begin{array}{ll}
      R_{2}(u,v)=&(-u+v-u^{2}v+uv^{2}+v^{2})/u=u \\
      Q_{2}(u,v)=&(-u^{2}+uv+3u^{2}v-2uv^{2}-u^{3}+v^{4}+v^{3}+u^{4}v^{2}-u^{4}v)/u^{2} \\
&+(3u^{3}v^{2}-2u^{3}v^{3}-4u^{2}v^{3}-uv^{3}+u^2v^{4}+2uv^{4})/u^{2}=v.\label{63}
\end{array}%
\right.
\end{equation}
The first equation $R_{2}(u,v)=u$ can be written in the form
\begin{equation}
-(v+1)(u+1)(u-v)=0. \label{ae}
\end{equation}
Out of the three solutions of $(\ref{ae})$, $u=-1$ and $u=v$ satisfy the pair of equations ($\ref{63}$). However $u=-1$ is the equation of period two orbits on the $(u,v)$-plane, since by formula $(\ref{1abc})$ we obtain $u+1=x+y+1=0$, in which the latter is the equation of the period two orbits in the $(x,y)$-plane (see \cite{1}).

We obtain neutral periodic orbits of period two when $u=-1$ in the formula $|\lambda_{2}(u,v)|=1$.
Now
\begin{displaymath}
\lambda_{2}(u,v)=\frac {\left({u}^{2}+u-v \right) \left({u}^{2}-u+v \right)}{{u}^{2}},
\end{displaymath}
and when $|\lambda_{2}(-1,v)|=1$ this implies that $-v^{2}-2v-1=0$ or $-v^{2}-2v+1=0$.
 The solutions of these equations are $(u,v)=(-1,-1+\sqrt{2})$, $(u,v)=(-1,-1-\sqrt{2})$ and $(u,v)=(-1,-1)$, which is a double root (the bifurcation point of period one). Thus $\{(-1,-1+\sqrt{2}),(-1,-1-\sqrt{2})\}$ is the real neutral cycle of period two and these points are also bifurcation points of period two on the $(u,v)$-plane (Figure $\ref{kuva9.1}$), as well as  $(x,y)=(\frac{-1-\sqrt{2}}{2},\frac{-1+\sqrt{2}}{2})$ and $(x,y)=(\frac{-1+\sqrt{2}}{2},\frac{-1-\sqrt{2}}{2})$ on the $(x,y)$-plane (Figure  $\ref{kuva9}$).
 Bifurcation points of period two are obtained also as intersection points of period four orbit, because the curve of period two bifurcates to curves of period four.

 We obtain all critical points of period two when the eigenvalue equation $|\lambda_{2}(u,v)|=0$ holds. In that case ${u}^{2}+u-v=0$ or ${u}^{2}-u+v=0$. When $u=-1$ we obtain cycles of critical points of period two on the $(u,v)$-plane as $(u,v)=(-1,0)$ and $(u,v)=(-1,-2)$ as well as on the $(x,y)$-plane as $(x,y)=(0,-1)$ and $(x,y)=(-1,0)$.
Because $(u,v)=(-1,-1/2)$ is an attracting point of period two ($|\lambda_{2}(-1,-1/2)|<1$), then
\begin{displaymath}
      V_{2} \setminus \{(u,v)|u=-1, v \in [-1-\sqrt{2},-1+\sqrt{2}]\},
\end{displaymath}
where $V_{2}=\{(u,v)|u+1=0, v \in \mathbb{R}\}$, belongs to the Julia set.

When we solve the equations $R_{n}(u,v)=u$ and $Q_{n}(u,v)=v$ to higher periods, the situation gets complicated very fast due to the growth of the degree in the equations. The following result is very useful in sections $4-6$.
 \begin{lause}
 Let $n \in \mathbb{N}\setminus{\{1\}}$. Any point of exact period $n+1$ satisfies
\begin{equation}
\left\{%
\begin{array}{ll}
      R_{0}R_{1}R_{2} \cdots R_{n}=&1\\
      Q_{0}Q_{1}Q_{2} \cdots Q_{n}=&1.\label{uviterkaava}
\end{array}%
\right.
\end{equation}
\end{lause}
\begin{proof}
Let us define $P_{k+1}=x_{k}^{2}+c$ and $P_{l+1}=x_{l}^{2}+c$, when $k,l \in \mathbb{N}$, $k\neq l$. Now we notice that $P_{k+1}-P_{l+1}=x_{k}^{2}-x_{l}^{2}=(x_{k}-x_{l})(x_{k}+x_{l})$, and so
\begin{equation}
x_{k}+x_{l}=\frac{P_{k+1}-P_{l+1}}{x_{k}-x_{l}}=\frac{x_{k+1}-x_{l+1}}{x_{k}-x_{l}}, \label{1a}
\end{equation}
for all $k, l \in \mathbb{N}$ such that $x_{k}\neq x_{l}$. Because $R_{0}=u=x_{0}+x_{1}$ and $Q_{0}=v=x_{0}+x_{2}$, and $x_{j+1}\neq x_{j}$ for all $j \in \{0,1,\ldots,n\}$, we get the next product by using ($\ref{1a}$), as follows:
\begin{eqnarray*}
R_{0}R_{1}R_{2} \cdots R_{n}&=&(x_{0}+x_{1})(x_{1}+x_{2})(x_{2}+x_{3})\cdots(x_{n}+x_{n+1})\\
&=&\Big(\frac{x_{1}-x_{2}}{x_{0}-x_{1}}\Big)\Big(\frac{x_{2}-x_{3}}{x_{1}-x_{2}}\Big)\Big(\frac{x_{3}-x_{4}}{x_{2}-x_{3}}\Big)\cdots\Big(\frac{x_{n+1}-x_{n+2}}{x_{n}-x_{n+1}}\Big)\\
&=&\frac{x_{n+1}-x_{n+2}}{x_{0}-x_{1}}.
\end{eqnarray*}
Let us assume that $x_{n+1}=P_{n+1}=P_{0}=x_{0}=x$. Then $x_{n+2}=P_{n+2}=x_{n+1}^{2}+c=x_{0}^{2}+c=x_{1}$, so
\begin{displaymath}
R_{0}R_{1}R_{2} \cdots R_{n}=\frac{x_{n+1}-x_{n+2}}{x_{0}-x_{1}}=\frac{x_{0}-x_{1}}{x_{0}-x_{1}}=1
\end{displaymath}
and the assertion follows.
Similarly, since $x_{j+2}\neq x_{j}$ for all $j \in \{0,1,\ldots,n\}$,
\begin{eqnarray*}
Q_{0}Q_{1}Q_{2} \cdots Q_{n}&=&(x_{0}+x_{2})(x_{1}+x_{3})(x_{2}+x_{4})\cdots(x_{n}+x_{n+2})\\
&=&\Big(\frac{x_{1}-x_{3}}{x_{0}-x_{2}}\Big)\Big(\frac{x_{2}-x_{4}}{x_{1}-x_{3}}\Big)\Big(\frac{x_{3}-x_{5}}{x_{2}-x_{4}}\Big)\cdots\Big(\frac{x_{n+1}-x_{n+3}}{x_{n}-x_{n+2}}\Big)\\
&=&\frac{x_{n+1}-x_{n+3}}{x_{0}-x_{2}}=1,
\end{eqnarray*}
when $x_{n+3}=x_{n+2}^{2}+c=(x_{n+1}^{2}+c)^{2}+c=(x_{0}^{2}+c)^{2}+c=x_{1}^{2}+c=x_{2}$.
\end{proof}

\section{Periodic orbits of period $3$}
Periodic points of period three satisfy the pair of equations
\begin{equation}
\left\{%
\begin{array}{ll}
      R_{3}(u,v)=&u \\
      Q_{3}(u,v)=&v,
\end{array}%
\right.  \label{8}
\end{equation}
where
\begin{displaymath}
\left\{%
\begin{array}{ll}
      R_{3}(u,v)=&\Big(\left( {v}^{2}-v \right) {u}^{4}+ \left( -2\,{v}^{3}+3\,{v}^{2}-v \right) {u}^{3}\\
      &+ \left( {v}^{4}-4\,{v}^{3}+{v}^{2}+v-1 \right) {u}^{2}+\left( 2\,{v}^{4}-{v}^{3}-{v}^{2}+v \right)u+ {v}^{4}+{v}^{3}\Big)/u^{2} \\ \\
      Q_{3}(u,v)=&\Big( \left( {v}^{4}-2\,{v}^{3}+{v}^{2} \right) {u}^{8}+ \left( -4\,{v}^{5}+10\,{v}^{4}-6\,{v}^{3}-{v}^{2}+v \right) {u}^{7}\\
&+\left( 6\,{v}^{6}-22\,{v}^{5}+17\,{v}^{4}+6\,{v}^{3}-8\,{v}^{2}+v \right) {u}^{6}\\
&+ \left( -4\,{v}^{7}+26\,{v}^{6}-30\,{v}^{5}-11\,{v}^{4}+23\,{v}^{3}-4\,{v}^{2}-3\,v+1 \right) {u}^{5}\\
&+ \left( {v}^{8}-16\,{v}^{7}+34\,{v}^{6}+6\,{v}^{5}-35\,{v}^{4}+8\,{v}^{3}+7\,{v}^{2}-3\,v-1 \right) {u}^{4}\\
&+ \left(4\,{v}^{8}-22\,{v}^{7}+ 6\,{v}^{6}+31\,{v}^{5}-11\,{v}^{4}-8\,{v}^{3}+5\,{v}^{2}+v \right) {u}^{3}\\
&+ \left( 6\,{v}^{8}-10\,{v}^{7}-15\,{v}^{6}+10\,{v}^{5}+6\,{v}^{4}-3\,{v}^{3} \right) {u}^{2}\\
&+ \left(4\,{v}^{8}+2\,{v}^{7}-6\,{v}^{6}-3\,{v}^{5}+{v}^{4} \right) u+{v}^{8}+2\,{v}^{7}+{v}^{6}\Big)/u^{4}.
\end{array}%
\right.
\end{displaymath}
On the other hand, we know that ($\ref{uviterkaava}$) is satisfied for period three orbits, so
\begin{equation}
1=R_{0}R_{1}R_{2}.\label{a8}
\end{equation}
Using formulas ($\ref{61}$) and ($\ref{63}$), it follows that the previous equation ($\ref{a8}$) can be factorized as
\begin{displaymath}
-{\frac { \left( uv+1+v \right)  \left( -{u}^{2}+{u}^{2}v-u{v}^{2}+uv+
u-{v}^{2} \right) }{u}}=0,
\end{displaymath}
where both polynomials are undivisible.
This equation is satisfied when
\begin{displaymath}
uv+1+v=0,
\end{displaymath}
or
\begin{displaymath}
 -{u}^{2}+{u}^{2}v-u{v}^{2}+uv+
u-{v}^{2} ={u}^{2}(v-1)+u(1+v-{v}^{2})-{v}^{2}
=0.
\end{displaymath}
Solutions of the latter equation are
\begin{displaymath}
u=\frac{-1-v+{v}^{2} \pm \sqrt{{v}^{4}+2\,{v}^{3}-5\,{v}^{2}+2\,v+1}}{2v-2},
\end{displaymath}
which do not satisfy the system of equations ($\ref{8}$), except at single points only (in the point $(1,1)$ and in the three points listed below). Instead, the solution
\begin{equation}
v=-\frac{1}{1+u}, \label{84}
\end{equation}
is the equation for the period three orbits, since by using formulas ($\ref{1abc}$) and ($\ref{1bcd}$) it follows that $(x+y^{2}+y-x^{2})(1+x+y)=-1$ on the $(x,y)$-plane (see \cite{1}). Moreover, we notice that the pair of equations
\begin{displaymath}
\left\{%
\begin{array}{l}
   uv+1+v=0   \\
     {u}^{2}(v-1)+u(1+v-{v}^{2})-{v}^{2}=0,
\end{array}%
\right.
\end{displaymath}
and the equation
\begin{displaymath}
\frac{-u^{4}-2u^{3}-u^{2}-u-1}{(u+1)^{2}}=\frac{-{u}^{3}-{u}^{2}-1}{u+1}=0
\end{displaymath}
are equivalent. Both of them have only one real solution. By a numerical approximation (rounded to six decimals) we obtain solutions
 \begin{displaymath}
\begin{array}{l}
 (u,v)\approx(-1.465571, 2.147899), \\ (u,v)\approx(0.232786-0.792552i, -0.573950-0.368989i),\\ (u,v)\approx(0.232786+0.792552i, -0.573950+0.368989i).
\end{array}
\end{displaymath}

 The equation ($\ref{84}$) for period three is of degree three on the $(x,y)$-plane for both variables.
There is a remarkable difference to the situation on the $(u,v)$-plane, where the same equation is only of degree one.

The eigenvalue of period three in the $(u,v)$-plane is
\begin{displaymath}
\lambda_{3}(u,v)=-{\frac {\left (-{u}^{2}-u+v\right )\left
({u}^{2}-u+v\right )\left (- {u}^{2}-u+v+2\,uv\right )}{{u}^{3}}}.
\end{displaymath}
Now we get neutral periodic points by solving the pair of equations
\begin{equation}
\left\{%
\begin{array}{l}
      1+v+vu=0 \\
      |\left ({u}^{2}+u-v\right )\left
({u}^{2}-u+v\right )\left (- {u}^{2}-u+v+2\,uv\right )/{u}^{3}|=1.\label{neutral3}
\end{array}%
\right.
\end{equation}
 Also in the case of period three, neutral periodic points are bifurcation points. The system of equations ($\ref{neutral3}$) produces neutral periodic points and also all cycles of bifurcation points of period three. All bifurcation points of period three (there are $18$ but the points $(u,v)=(-1/2+i\sqrt{3}/2, -1/2+i\sqrt{3}/2)$ and $(u,v)=(-1/2-i\sqrt{3}/2, -1/2-i\sqrt{3}/2)$ are roots of order three) are shown in the Table $4.1$. Out of them six are real (two cycles of three points) and eight are complex points. We see the real points in the Figure $\ref{kuva8.1}$ by intersection points of the curves of the period three and the eigenvalue $|\lambda_{3}(u,v)|=1$. Also we can conclude what parts of period three curves belong to the Fatou set in the following way. Attractive parts of the curves are between the neutral periodic points, which we can see also in the bifurcation diagram (Figure $\ref{kuva19}$), as well as in the case of the period 1 and 2. Naturally the rest of the points of the curves (repulsive parts) belong to the Julia set. The roots of order three correspond to the point at which the main cardioid bifurcates into the rabbit and co-rabbit components of the classical Mandelbrot set. Attracting periodic points form well known windows of period three in the bifurcation diagram, which we see clearly in the lowest part in pictures both on the $(u,v)$-plane (Figure $\ref{kuva19}$) and the $(x,y)$-plane (Figure $\ref{kuva18}$).
\begin{table}[h]
\begin{center}
Table $4.1$.
\end{center}
\begin{center}
 Bifurcation points of period three on the $(u,v)$-plane.
\end{center}
\begin{displaymath}
\begin{array}{|ll|}
   \hline
  u \approx -1.713318 & v \approx 1.401899 \\\hline  u \approx -1.801938 & v \approx 1.246980  \\\hline  u \approx 1.246980 & v \approx -0.445042 \\\hline u \approx 1.401899 & v \approx -0.416337 \\\hline  u \approx -0.416337 & v \approx -1.713318 \\\hline  u \approx -0.445042 & v \approx -1.801938  \\\hline u \approx -0.895298+1.448231i & v \approx -0.049661+0.686907i \\\hline u \approx -0.691163+0.499574i &  v \approx -0.895298+1.448231i \\\hline u \approx -0.049661+0.686907i & v \approx -0.691163+0.499574i \\\hline  u \approx -0.049661-0.686907i & v \approx -0.691163-0.499574i \\\hline u \approx -0.691163-0.499574i & v \approx -0.895298-1.448231i \\\hline u \approx -0.895298-1.448231i & v \approx -0.049661-0.686907i  \\\hline u = -1/2+i\sqrt{3}/2 & v = -1/2+i\sqrt{3}/2 \\\hline
  u = -1/2-i\sqrt{3}/2 & v = -1/2-i\sqrt{3}/2 \\\hline
\end{array}
\end{displaymath}
\end{table}

We obtain all critical points or super-attractive points as solutions of the pair of equations
 \begin{equation}
\left\{%
\begin{array}{l}
      1+v+vu=0 \\
      \left ({u}^{2}+u-v\right )\left
({u}^{2}-u+v\right )\left (- {u}^{2}-u+v+2\,uv\right )=0.\label{2bb}
\end{array}%
\right.
\end{equation}
 We know by the definition on the page 5, that the central critical points of the $(u,v)$-plane lie on the curve $v=u^{2}+u$. Thus we obtain one point at each critical cycle of period three as a solution of the pair of equations
 \begin{equation}
\left\{%
\begin{array}{l}
      1+v+vu=0 \\
      v=u^{2}+u. \label{2abb}
\end{array}%
\right.
\end{equation}
 A great benefit of this method is that the system of equations is of significantly lower degree than the system ($\ref{2bb}$). The pair of equations ($\ref{2abb}$) is equivalent with the equation
\begin{displaymath}
u^{3}+2u^{2}+u+1=0,
\end{displaymath}
and a numerical approximation of its solutions are
\begin{displaymath}
\begin{array}{l}
(u,v)\approx(-1.754878,1.324718),\\ (u,v)\approx(-0.122561-0.744862i, -0.662359-0.562280i), \\ (u,v)\approx(-0.122561+0.744862i, -0.662359+0.562280i).
\end{array}
\end{displaymath}
All critical points of period three orbits are presented in the Table $4.2$. The cycles are naturally obtained by iteration of the function~$G$.

\begin{table}[h]
\begin{center}
Table $4.2$.
\end{center}
\begin{center}
  Critical points of period three orbits on the $(u,v)$-plane.
\end{center}
\begin{displaymath}
\begin{array}{|ll|}
    \hline
  u \approx -1.754878 & v \approx 1.324718 \\\hline u \approx -0.430160 & v \approx -1.754878 \\\hline u \approx 1.324718 & v \approx -0.430160  \\\hline  u \approx -0.122561-0.744862i & v \approx -0.662359-0.562280i \\\hline u \approx -0.784920-1.307141i & v \approx -0.122561-0.744862i \\\hline u \approx -0.662359-0.562280i & v \approx -0.784920-1.307141i \\\hline u \approx -0.122561+0.744862i & v \approx -0.662359+0.562280i \\\hline u \approx -0.784920+1.307141i & v \approx -0.122561+0.744862i \\\hline u \approx -0.662359+0.562280i & v \approx -0.784920+1.307141i \\\hline
\end{array}
\end{displaymath}
\end{table}

\section{Periodic orbits of period $4$}

Periodic points of period four satisfy the pair of equations
\begin{equation}
\left\{%
\begin{array}{ll}
      R_{4}(u,v)=&u \\
      Q_{4}(u,v)=&v.
\end{array}%
\right.  \label{91}
\end{equation}
In addition,  we know that ($\ref{uviterkaava}$) is satisfied by period four orbits, when
\begin{equation}
1=uR_{1}R_{2}R_{3}. \label{ab}
\end{equation}
When we substitute these formulas of $R_{1}$, $R_{2}$ and $R_{3}$ into ($\ref{ab}$) and factorize, we get
\begin{equation}
\frac{A_{4}B_{4}C_{4}}{{u}^{3}}=0,\label{94}
\end{equation}
where
\begin{displaymath}
\begin{array}{ll}
      A_{4}=& -(1+u),\\
      B_{4}=&-{u}^{2}v+{u}^{2}{v}^{2}-u+uv+u{v
}^{2}-{v}^{2}-{v}^{3}-u{v}^{3}, \\
      C_{4}=&-{u}^{4}v+{u}^{4}{v}^{2
}-2\,{u}^{3}{v}^{3}+3\,{u}^{3}{v}^{2}+{u}^{2}{v}^{4}-4\,{u}^{2}{v}^{3}
+2\,{u}^{2}v-{u}^{2}+2\,u{v}^{4}-u{v}^{3} \\
&-2\,u{v}^{2}+{v}^{4}+{v}^{3}.
\end{array}
\end{displaymath}
This formula includes orbits of period two. Equation ($\ref{94}$) is also satisfied when $B_{4}=0$ or $C_{4}=0$.
On the other hand, on the $(x,c)$-plane we know that the equation of period four orbits is
\begin{eqnarray*}
&&((((x^{2}+c)^{2}+c)^{2}+c)^{2}+c-x)/((x^{2}+c)^{2}+c-x)\\
&&={x}^{12}+6\,c{x}^{10}+{x}^{9}+ \left( 15\,{c}^{2}+3\,c \right) {x}^{8}+4\,c{x}^{7}+ \left( 20\,{c}^{3}+12\,{c}^{2}+1 \right) {x}^{6}\\
&&+ \left( 6\,{c}^{2}+2\,c \right) {x}^{5}+ \left( 15\,{c}^{4}+18\,{c}^{3}+3\,{c}^{2}+4\,c \right) {x}^{4}+ \left( 4\,{c}^{3}+4\,{c}^{2}+1 \right) {x}^{3}\\
&&+ \left( 6\,{c}^{5}+12\,{c}^{4}+6\,{c}^{3}+5\,{c}^{2}+
c \right) {x}^{2}+ \left( {c}^{4}+2\,{c}^{3}+{c}^{2}+2\,c \right) x\\
&&+{c}^{6}+3\,{c}^{5}+3\,{c}^{4}+3\,{c}^{3}+2\,{c}^{2}+1=0.
\end{eqnarray*}
Now when we replace variables $x$ and $c$ in the previous formula by formulas $(\ref{4})$ and $(\ref{4.1})$, we obtain the equation of period four on the $(u,v)$-plane. In conclusion, the equation of period four curves is $B_{4}=0$, so
\begin{displaymath}
-{u}^{2}v+{u}^{2}{v}^{2}-u+uv+u{v
}^{2}-{v}^{2}-{v}^{3}-u{v}^{3}=0,
\end{displaymath}
which can be written as
\begin{equation}
u^{2}(-v^{2}+v)+u(v^{3}-v^{2}-v+1)+v^{3}+v^{2}=0 \label{11}
\end{equation}
and solved explicitly for $u$ as
\begin{displaymath}
u=\frac{v^{3}-v^{2}-v+1\pm
\sqrt{v^{6}+2v^{5}-v^{4}-v^{2}-2v+1}}{2(v^{2}-v)}. \label{12}
\end{displaymath}
On the $(x,y)$-plane the degree of the equation of period four orbits is six for both variables, thus it is impossible to solve them explicitly.
The equation of period four orbits has been presented in the literature before now explicitly only twice; T.~Erkama~\cite{1} presented nearly the same form at $2006$ and also Morton \cite{2} presented a corresponding solution, but in a more complicated form at $1998$. Morton's solution of period four orbits equation by using model $P(x)=x^{2}+c$ is
\begin{displaymath}
x=\frac{w}{2}\pm \frac{\triangle}{2(w^{3}-w)} ,
\end{displaymath}
where
\begin{displaymath}
\triangle^{2}=(w^{4}-1)(w^{2}+2w-1),
\end{displaymath}
and
\begin{displaymath}
w=x^{4}+2cx^{2}+x+c+c^{2}.
\end{displaymath}

The eigenvalue of period four on the $(u,v)$-plane is
\begin{displaymath}
\lambda_{4}(u,v)={\frac {a_{4}b_{4}c_{4}d_{4}}{{u}^{4}}},
\end{displaymath}
where
\begin{eqnarray*}
a_{4}&=&{u}^{2}+u-v ,\\
b_{4}&=&{u}^{2}-u+v,\\
c_{4}&=&2\,uv+v-u-{u}^{2},\\
d_{4}&=&2\,u{v}^{2}+2\,{v}^{2}+v-2\,
uv-u-2\,{u}^{2}v+{u}^{2}.
\end{eqnarray*}
We obtain bifurcation points of period four as solutions of the system of equations
\begin{displaymath}
\left\{%
\begin{array}{l}
     u^{2}(-v^{2}+v)+u(v^{3}-v^{2}-v+1)+v^{3}+v^{2}=0 \\
      |\lambda_{4}(u,v)|=1,
\end{array}%
\right.
\end{displaymath}
which has $56$ solutions all together with multiplicity, but only $42$ distinct solutions (using the Singular program \cite{100}). Real bifurcation points can be seen in Figure $\ref{kuva9.4}$, where they are the intersection points of curves of period four and eigenvalues $|\lambda_{4}(u,v)|=1$. The central critical points are obtained as solutions of the pair of equations
\begin{displaymath}
\left\{%
\begin{array}{l}
     u^{2}(-v^{2}+v)+u(v^{3}-v^{2}-v+1)+v^{3}+v^{2}=0 \\
      v=u^{2}+u,
\end{array}%
\right.
\end{displaymath}
which is equivalent with the equation
\begin{displaymath}
{u}^{7}+3\,{u}^{6}+3\,{u}^{5}+3\,{u}^{4}+2\,{u}^{3}+u=0.
\end{displaymath}
Since $(u,v)=(0,0)$ is the point of indeterminacy (see Theorem $\ref{lause1.2}$), numerical approximations of its solutions are
 \begin{displaymath}
\begin{array}{l}
 (u,v)\approx(-1.940800, 1.825904), \\ (u,v)\approx(-1.310703,0.407239), \\  (u,v)\approx(0.282271-0.530061i, 0.080984-0.829303i),\\ (u,v)\approx(0.282271+0.530061i, 0.080984+0.829303i),\\ (u,v)\approx(-0.156520-1.032247i, -1.197556-0.709112i),\\ (u,v)\approx(-0.156520+1.032247i, -1.197556+0.709112i).
\end{array}
\end{displaymath}
In the Table $5.1$ all critical points of period four have been presented. There are $24$ points all together, six cycles out of which two is real.


\section{Periodic orbits of period $5$}

Orbits of period five satisfy a pair of equations
\begin{equation}
\left\{%
\begin{array}{ll}
      R_{5}(u,v)=&u\\
      Q_{5}(u,v)=&v.\label{101}
\end{array}%
\right.
\end{equation}
According to ($\ref{uviterkaava}$) the equation
\begin{equation}
1=uR_{1}R_{2}R_{3}R_{4} \label{104}
\end{equation}
contains orbits of period five.
By the formulas ($\ref{6}$) we obtain $R_{4}$ by iteration and using the terms computed earlier. Since $(R_{4},Q_{4})=G(R_{3},Q_{3})$, it follows that  \\
\begin{eqnarray}
R_{4}&=&\Big( \left( {v}^{4}-2\,{v}^{3}+{v}^{2} \right) {u}^{8}+ \left( -4\,{v}^{5}+10\,{v}^{4}-6\,{v}^{3}-{v}^{2}+v \right) {u}^{7} \nonumber \\
&&+\left(6\,{v}^{6} -22\,{v}^{5}+17\,{v}^{4}+6\,{v}^{3}-7\,{v}^{2} \right) {u}^{6} \nonumber \\
&&+\left(-4\,{v}^{7}+26\,{v}^{6}-30\,{v}^{5}-11\,{v}^{4}+21\,{v}^{3}-{v}^{2}-2\,v \right) {u}^{5} \nonumber \\
&&+\left( {v}^{8}-16\,{v}^{7}+34\,{v}^{6}+6\,{v}^{5}-34\,{v}^{4}+4\,{v}^{3}+6\,{v}^{2}-1 \right) {u}^{4} \label{515} \\
&&+\left( +4\,{v}^{8}-22\,{v}^{7}+6\,{v}^{6}+31\,{v}^{5}-9\,{v}^{4}-9\,{v}^{3}+2\,{v}^{2}+v\right) {u}^{3} \nonumber \\
&&+\left(6\,{v}^{8}-10\,{v}^{7}-15\,{v}^{6}+10\,{v}^{5}+7\,{v}^{4} -2\,{v}^{3}\right) {u}^{2} \nonumber \\
&&+\left( 4\,{v}^{8}+2\,{v}^{7}-6\,{v}^{6}-3\,{v}^{5}+{v}^{4} \right)u+{v}^{8}+2\,{v}^{7}+{v}^{6} \Big)/{u}^{4}.\nonumber
\end{eqnarray}

\begin{table}[h]
\begin{center}
Table $5.1$.
\end{center}
\begin{center}
  Critical points of period four on the $(u,v)$-plane.
\end{center}
\begin{displaymath}
\begin{array}{|ll|}
       \hline
u\approx-1.940800& v\approx1.825904 \\\hline u\approx-0.114896& v\approx-0.547673\\\hline u\approx3.219030& v\approx1.825904 \\\hline u\approx1.393126& v\approx-0.547674 \\\hline  u\approx-1.310703& v\approx0.407239 \\\hline u\approx-0.903464& v\approx-2.455562 \\\hline u\approx-0.737620& v\approx0.407239 \\\hline u\approx-1.144859& v\approx-2.455562 \\\hline u\approx0.282271-0.530061i& v\approx0.080984-0.829303i \\\hline  u\approx0.363256-1.359363i& v\approx-0.116641-1.194442i \\\hline u\approx-0.317929-1.493684i& v\approx0.080984-0.829303i \\\hline u\approx-0.398913-0.664382i& v\approx-0.116641-1.194442i \\\hline u\approx0.282271+0.530061i& v\approx0.080984+0.829303i \\\hline u\approx0.363256+1.359363i& v\approx-0.116641+1.194442i \\\hline u\approx-0.317929+1.493684i& v\approx0.080984+0.829303i \\\hline u\approx-0.398913+0.664382i& v\approx-0.116641+1.194442i \\\hline
u\approx-0.156520-1.032247i& v\approx-1.197556-0.709112i \\\hline  u\approx-1.354076-1.741359i& v\approx0.618259-0.366092i \\\hline u\approx-0.422776-0.042957i& v\approx-1.197556-0.709112i \\\hline u\approx0.774779+0.666155i& v\approx0.618259-0.366092i \\\hline
u\approx-0.156520+1.032247i & v\approx-1.197556+0.709112i \\\hline u\approx-1.354076+1.741359i& v\approx0.618259+0.366092i \\\hline u\approx-0.422776+0.042957i& v\approx-1.197556+0.709112i \\\hline u\approx0.774779-0.666155i& v\approx0.618259+0.366092i \\\hline
\end{array}
\end{displaymath}
\end{table}

Next substitute formulas for $R_{1}$, $R_{2}$, $R_{3}$ and $R_{4}$ by using ($\ref{61}$), ($\ref{63}$), ($\ref{8}$) and ($\ref{515}$) to equation ($\ref{104}$) and factorize this form to obtain
\begin{equation}
\frac{P_{5}L_{5}}{u^{7}}=0,\label{105}
\end{equation}
where both polynomials are undivisible and
\begin{eqnarray}
 P_{5}&=&\left( {v}^{4}-2\,{v}^{3}+{v}^{2} \right) {u}^{7}+ \left( -3\,{v}^{5}+8\,{v}^{4}-5\,{v}^{3}-{v}^{2}+v \right) {u}^{6} \nonumber \\
 &&+\left(3\,{v}^{6}-14\,{v}^{5}+12\,{v}^{4}+5\,{v}^{3}-7\,{v}^{2}+v \right) {u}^{5} \nonumber \\
 &&+\left(-{v}^{7}+12\,{v}^{6}-18\,{v}^{5}-6\,{v}^{4}+16\,{v}^{3}-3\,{v}^{2}-2\,v \right) {u}^{4} \nonumber \\
 &&+\left(-4\,{v}^{7}+16\,{v}^{6}-19\,{v}^{4}+5\,{v}^{3}+4\,{v}^{2}-2\,v-1 \right) {u}^{3}  \\ \label{123}
 &&+ \left(-6\,{v}^{7}+6\,{v}^{6}+12\,{v}^{5}-6\,{v}^{4}-4\,{v}^{3}+2\,{v}^{2} \right) {u}^{2} \nonumber \\
 &&+\left(-4\,{v}^{7}-3\,{v}^{6}+4\,{v}^{5}+2\,{v}^{4}-{v}^{3}\right)u-{v}^{7}-2\,{v}^{6}-{v}^{5}, \nonumber
\end{eqnarray}


\begin{eqnarray*}
L_{5}&=&\left( {v}^{4}-2\,{v}^{3}+{v}^{2} \right) {u}^{8}+ \left( -4\,{v}^{5}+10\,{v}^{4}-6\,{v}^{3}-{v}^{2}+v \right) {u}^{7}  \\
&&+\left(6\,{v}^{6}-22\,{v}^{5}+17\,{v}^{4}+6\,{v}^{3}-8\,{v}^{2}+v \right) {u}^{6}  \\
&&+\left(-4\,{v}^{7}+26\,{v}^{6}-30\,{v}^{5}-11\,{v}^{4}+23\,{v}^{3}-4\,{v}^{2}-2\,v+1 \right) {u}^{5} \\
&&+\left({v}^{8}-16\,{v}^{7}+34\,{v}^{6}+6\,{v}^{5}-35\,{v}^{4}+8\,{v}^{3}+6\,{v}^{2}-2\,v-1 \right) {u}^{4} \\
&&+\left(4\,{v}^{8}-22\,{v}^{7}+6\,{v}^{6}+31\,{v}^{5}-11\,{v}^{4}-8\,{v}^{3}+4\,{v}^{2} \right) {u}^{3} \\
&&+\left(6\,{v}^{8}-10\,{v}^{7}-15\,{v}^{6}+10\,{v}^{5}+6\,{v}^{4}-3\,{v}^{3} \right) {u}^{2}  \\
&&+\left(4\,{v}^{8}+2\,{v}^{7}-6\,{v}^{6}-3\,{v}^{5}+{v}^{4} \right)u+{v}^{8}+2\,{v}^{7}+{v}^{6}.
\end{eqnarray*}
Still we know that the equation of the period five on the $(x,c)$-plane is
\begin{displaymath}
(((((x^{2}+c)^{2}+c)^{2}+c)^{2}+c)^{2}+c-x)/(x^{2}+c-x)=0.
\end{displaymath}
 By using formulas $(\ref{4})$ and $(\ref{4.1})$ with the previous formula, we see that $P_{5}=0$ is the equation of period five orbits on the $(u,v)$-plane.
 The curve $P_{5}(u,v)=0$ is of degree seven in both variables. Respectively on the $(x,y)$-plane the degree of period five equation is fifteen in both variables and it is of the form
 \begin{eqnarray}
 &&c_{1}{x}^{15}+ c_{2}{x}^{14}+ c_{3}{x}^{13}+ c_{4}{x}^{12}+ c_{5}{x}^{11}+ c_{6}{x}^{10}+ c_{7}{x}^{9}+ c_{8}{x}^{8}+ c_{9}{
x}^{7} \\ \nonumber
&&+ c_{10}{x}^{6}+
 c_{11}{x}^{5}+ c_{12}{x}^{4}+c_{13}{x}^{3}+
 c_{14}{x}^{2}+c_{15}x+c_{16}=0,  \label{12.2}
\end{eqnarray}
where
\begin{eqnarray*}
c_{1}&=& 1  \\
c_{2}&=& y \\
c_{3}&=&  -7\,{y}^{2}-8\,y-4 \\
c_{4}&=&  -7\,{y}^{3}-8\,{y}^{2}-4\,y\\
c_{5}&=&  21\,{y}^{4}+48\,{y}^{3}+48\,{y}^{2}+28\,y+6 \\
c_{6}&=&   21\,{y}^{5}+48\,{y}^{4}+48\,{y}^{3}+28\,{y}^{2}+6\,y\\
c_{7}&=&  -35\,{y}^{6}-120\,{y}^{5}-180\,{y}^{4}-172\,{y}^{3}-102\,{y}^{2}-36\,y-6\\
c_{8}&=&   -35\,{y}^{7}-120\,{y}^{6}-180\,{y}^{5}-172\,{y}^{4}-102\,{y}^{3}-36\,{y}^{2}-6\,y\\
c_{9}&=&  35\,{y}^{8}+160\,{y}^{7}+320\,{y}^{6}+408\,{y}^{5}+364\,{y}^{4}+224\,{y}^{3}+100\,{y}^{2}+30\,y+5\\
c_{10}&=&  35\,{y}^{9}+160\,{y}^{8}+320\,{y}^{7}+408\,{y}^{6}+364\,{y}^{5}+224\,{y}^{4}+100\,{y}^{3}+30\,{y}^{2}+5\,y\\
c_{11}&=& -21\,{y}^{10}-120\,{y}^{9}-300\,{y}^{8}-472\,{y}^{7}-540\,{y}^{6}-456\,{y}^{5}-296\,{y}^{4}-154\,{y}^{3}\\
&&-63\,{y}^{2}-20\,y-2\\
c_{12}&=& -21\,{y}^{11}-120\,{y}^{10}-300\,{y}^{9}-472\,{y}^{8}-540\,{y}^{7}-456\,{y}^{6}-296\,{y}^{5}-154\,{y}^{4}\\
&&-63\,{y}^{3}-20\,{y}^{2}-2\,y\\
c_{13}&=& 7\,{y}^{12}+48\,{y}^{11}+144\,{y}^{10}+268\,{y}^{9}+366\,{y}^{8}+384\,{y}^{7}+316\,{y}^{6}+218\,{y}^{5}\\
&&+127\,{y}^{4}+64\,{y}^{3}+28\,{y}^{2}+6\,y+1 \\
c_{14}&=& 7\,{y}^{13}+48\,{y}^{12}+144\,{y}^{11}+268\,{y}^{10} +366\,{y}^{9}+384\,{y}^{8}+316\,{y}^{7}+218\,{y}^{6}\\
&&+127\,{y}^{5}+64\,{y}^{4}+28\,{y}^{3}+6\,{y}^{2}+y \\
c_{15}&=& -{y}^{14}-8\,{y}^{13}-28\,{y}^{12}-60\,{y}^{11}-94\,{y}^{10}-116\,{y}^{9}-114\,{y}^{8}-94\,{y}^{7}-69\,{y}^{6}\\
&&-44\,{y}^{5}-26\,{y}^{4}-14\,{y}^{3}-5\,{y}^{2}-2\,y-1\\
c_{16}&=& -{y}^{15}-8\,{y}^{14}-28\,{y}^{13}-60\,{y}^{12}-94\,{y}^{11}-116\,{y}^{10}-114\,{y}^{9}-94\,{y}^{8}-69\,{y}^{7}\\
&&-44\,{y}^{6}-26\,{y}^{5}-14\,{y}^{4}-5\,{y}^{3}-2\,{y}^{2}-y-1.
\end{eqnarray*}

  It is well known that orbits of period five are not solvable explicitly.
The eigenvalue of period five on the $(u,v)$-plane is
\begin{equation}
\lambda_{5}(u,v)=\frac{a_{5}b_{5}c_{5}d_{5}e_{5}}{u^{6}},\label{105.1}
\end{equation}
where
\begin{eqnarray*}
a_{5}&=&u^2+u-v\\
b_{5}&=&u^2-u+v\\
c_{5}&=&-{u}^{2}+ \left( 2\,v-1 \right) u+v\\
d_{5}&=&\left( -2\,v+1 \right) {u}^{2}+ \left( 2\,{v}^{2}-2\,v-1 \right) u+2\,{v}^{2}+v\\
e_{5}&=&\left( 2\,{v}^{2}-2\,v \right) {u}^{4}+ \left( -4\,{v}^{3}+6\,{v}^{2}-1 \right) {u}^{3}+ \left( 2\,{v}^{4}-8\,{v}^{3}+4\,v-1 \right) {u}^{2}\\
&&+ \left( 4\,{v}^{4}-2\,{v}^{3}-4\,{v}^{2}+v \right) u+2\,{v}^{4}+2\,{v}^{3}.
\end{eqnarray*}
The bifurcation points are produced by the pair of equations
\begin{displaymath}
\left\{%
\begin{array}{ll}
      P_{5}(u,v)&=0\\
      |\lambda_{5}(u,v)|&=1,
\end{array}%
\right.
\end{displaymath}
which have altogether $194$ solutions with multiplicity, but only $138$ distinct solutions. Real bifurcation points of period five are observed in the Figure $\ref{kuva10.1}$.
\begin{table}[h]
\begin{center}
Table $6.1$.
\end{center}
\begin{center}
 Real critical points of period five on the $(u,v)$-plane.
\end{center}
\begin{displaymath}
\begin{array}{|ll|}
       \hline

                   u\approx-1.625414& v\approx1.016556 \\\hline  u\approx-0.608858 & v\approx -2.217441  \\\hline      u\approx0.424529 & v\approx-0.258361  \\\hline u\approx-1.866945& v\approx-0.592028  \\\hline              u\approx-1.274917& v\approx-2.900331  \\\hline u\approx-1.860783& v\approx1.601729\\\hline
                     u\approx-0.259053& v\approx-1.156029\\\hline
                     u\approx 2.306483& v\approx0.237624\\\hline u\approx-0.659352& v\approx0.704753\\\hline  u\approx-1.364105& v\approx-3.224888\\\hline u\approx-1.985424& v\approx1.956485\\\hline
                        u\approx-0.028939& v\approx-0.143014\\\hline     u\approx 3.798895& v\approx3.365536\\\hline
                       u\approx3.251461& v\approx1.842410\\\hline  u\approx 1.409051& v\approx-0.576373\\\hline
\end{array}
\end{displaymath}
\end{table}

The central critical points are obtained as solutions of the pair of equations of the critical curve $v=u^{2}+u$ and the period five orbits
 \begin{displaymath}
\left\{%
\begin{array}{l}
     P_{5}(u,v)=0 \\
      v=u^{2}+u,
\end{array}%
\right.
\end{displaymath}
which is equivalent with the equation
 \begin{eqnarray}
&&{u}^{18}+8\,{u}^{17}+28\,{u}^{16}+60\,{u}^{15}+94\,{u}^{14}+116\,{u}^{13}+114\,{u}^{12}+94\,{u}^{11}+69\,{u}^{10}  \nonumber \\
&&+44\,{u}^{9}+26\,{u}^{8}+14\,{u}^{7}+5\,{u}^{6}+2\,{u}^{5}+{u}^{4}+{u}^{3}=0. \nonumber
\end{eqnarray}
The real central critical points (Figure $\ref{kuva140}$) are (see \cite[p. 106]{3} or \cite{4})
\begin{displaymath}
\begin{array}{c}
(u,v)\approx(-1.625414, 1.016556), \\ (u,v)\approx(-1.860783, 1.601729), \\  (u,v)\approx(-1.985424, 1.956485).
\end{array}
\end{displaymath}
Table $6.1$ includes these points with their cycles. The complex solutions are
\begin{displaymath}
\begin{array}{l}
 (u,v)\approx(0.379514+0.334932i,0.411365+0.589155i),\\ (u,v)\approx(0.359259+0.642514i, 0.075503+1.104172i),\\ (u,v)\approx(-0.044212+0.986581i, -1.015600+0.899343i),\\ (u,v)\approx(-.198042+1.100270i, -1.369414+0.664470i),\\ (u,v)\approx(-0.504340+0.562766i, -0.566686-0.004885i),\\ (u,v)\approx(-1.256368+0.380321i, 0.177448-0.575325i),\\ (u,v)\approx(-1.256368-0.380321i, 0.177448+0.575325i),\\ (u,v)\approx(-0.504340-0.562766i, -0.566686+0.004885i),\\ (u,v)\approx(-0.198042-1.100270i, -1.369414-0.664470i),\\ (u,v)\approx(-0.044212-0.986581i, -1.015600-0.899343i),\\ (u,v)\approx(0.359259-0.642514i, 0.075503-1.104172i),\\ (u,v)\approx(0.379514-0.334932i, 0.411365-0.589155i).
\end{array}
\end{displaymath}
There are altogether $75$ critical points ($(u,v)=(0,0)$ is the point of indeterminacy).

\clearpage


\vspace*{1.5cm}
\begin{center}
''The world around us is very complicated. The tools at our
disposal to describe it are very weak.''
\end{center}
\vspace*{0.1cm}
\begin{center}
Benoit Mandelbrot
\end{center}



\clearpage
\section{Appendix}
Pictures of periodic orbit curves, curves of eigenvalues and bifurcation diagrams on the $(u,v)$-plane and on the $(x,y)$-plane.

\begin{figure}[h!]
\begin{center}
\includegraphics[width=0.57\textwidth,angle=270]{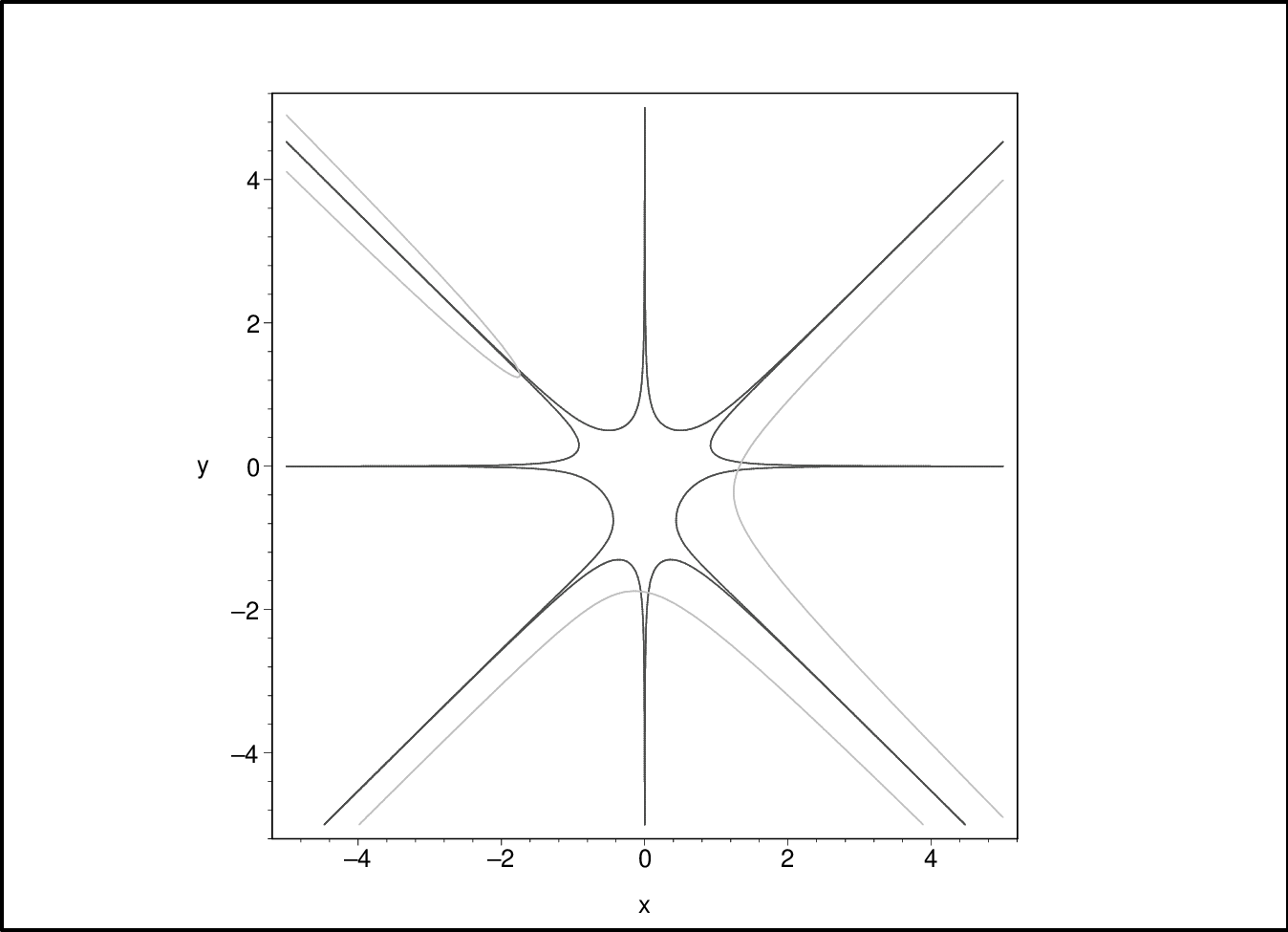}
\caption{Curves of period three orbits and eigenvalues $|\lambda_{3}(x,y)|=1$.} \label{kuva8}
\end{center}
\end{figure}
\begin{figure}[h!]
\begin{center}
\includegraphics[width=0.57\textwidth,angle=270]{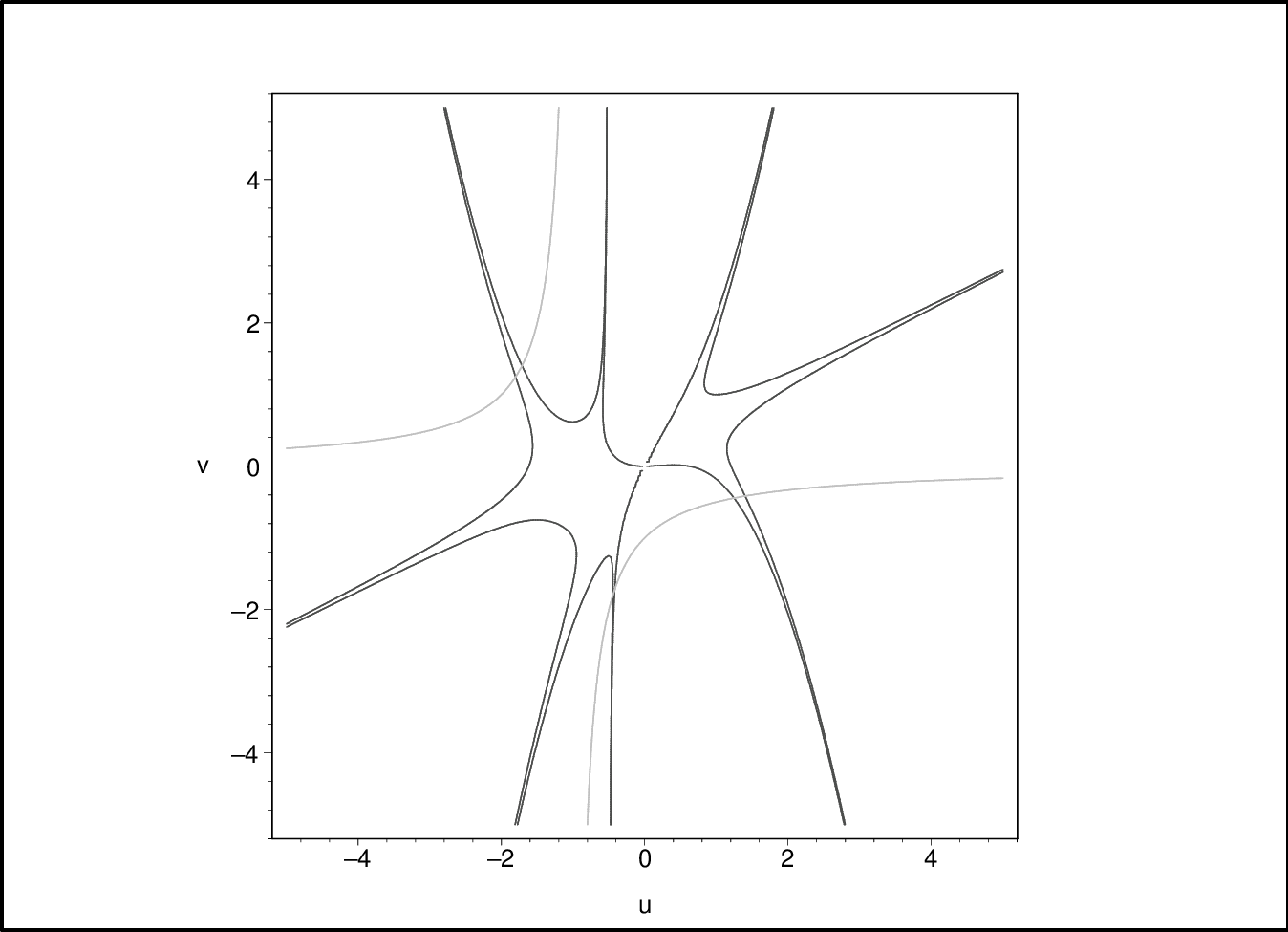}
\caption{Curves of period three orbits and eigenvalues $|\lambda_{3}(u,v)|=1$.} \label{kuva8.1}
\end{center}
\end{figure}
\clearpage

\clearpage

\begin{figure}[h!]
\begin{center}
\includegraphics[width=0.6\textwidth,angle=270]{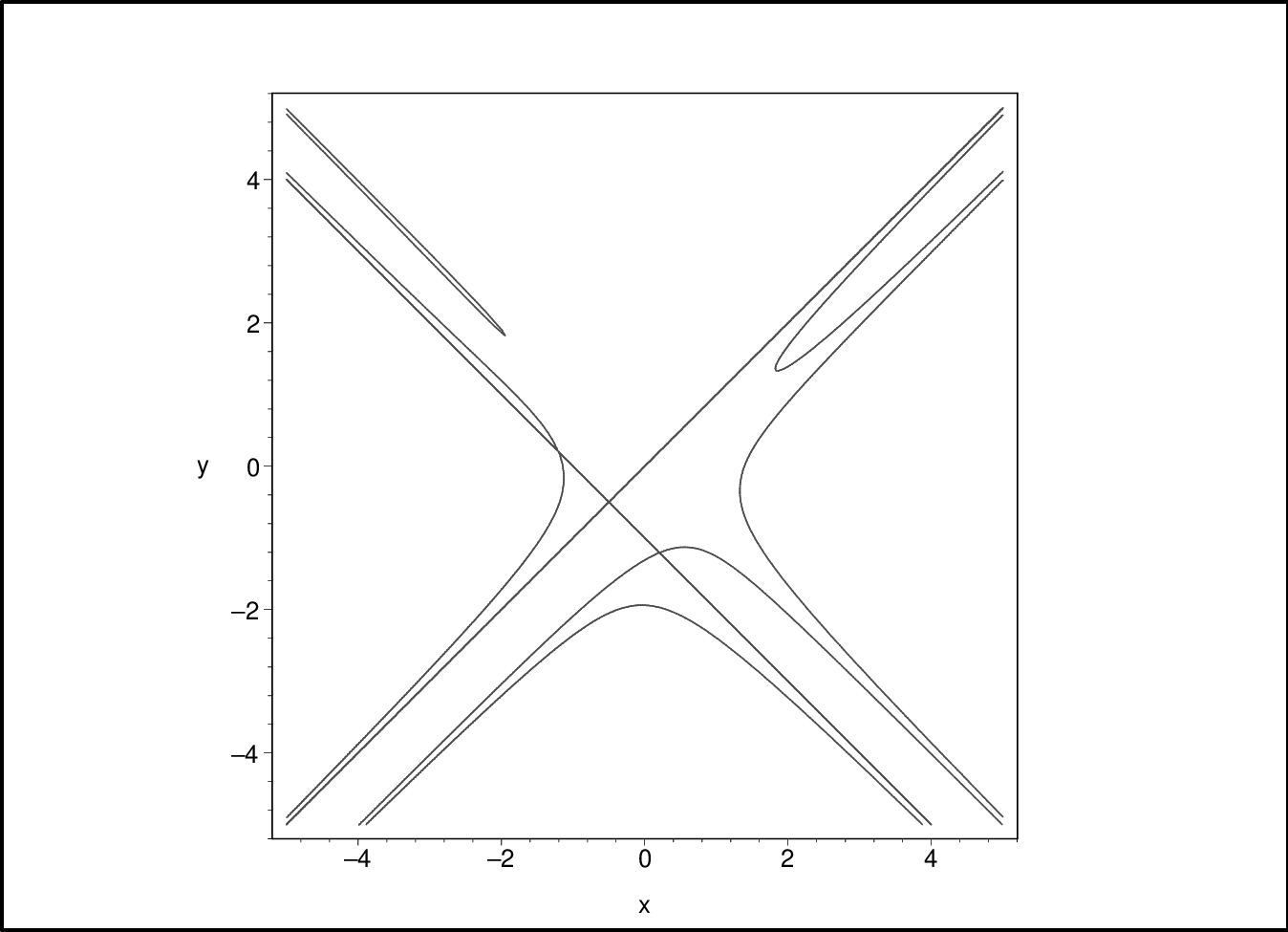}
\caption{Curves of period $1$, $2$ and $4$ orbits on the $(x,y)$-plane.}
\label{kuva9}
\end{center}
\end{figure}
\begin{figure}[h!]
\begin{center}
\includegraphics[width=0.6\textwidth,angle=270]{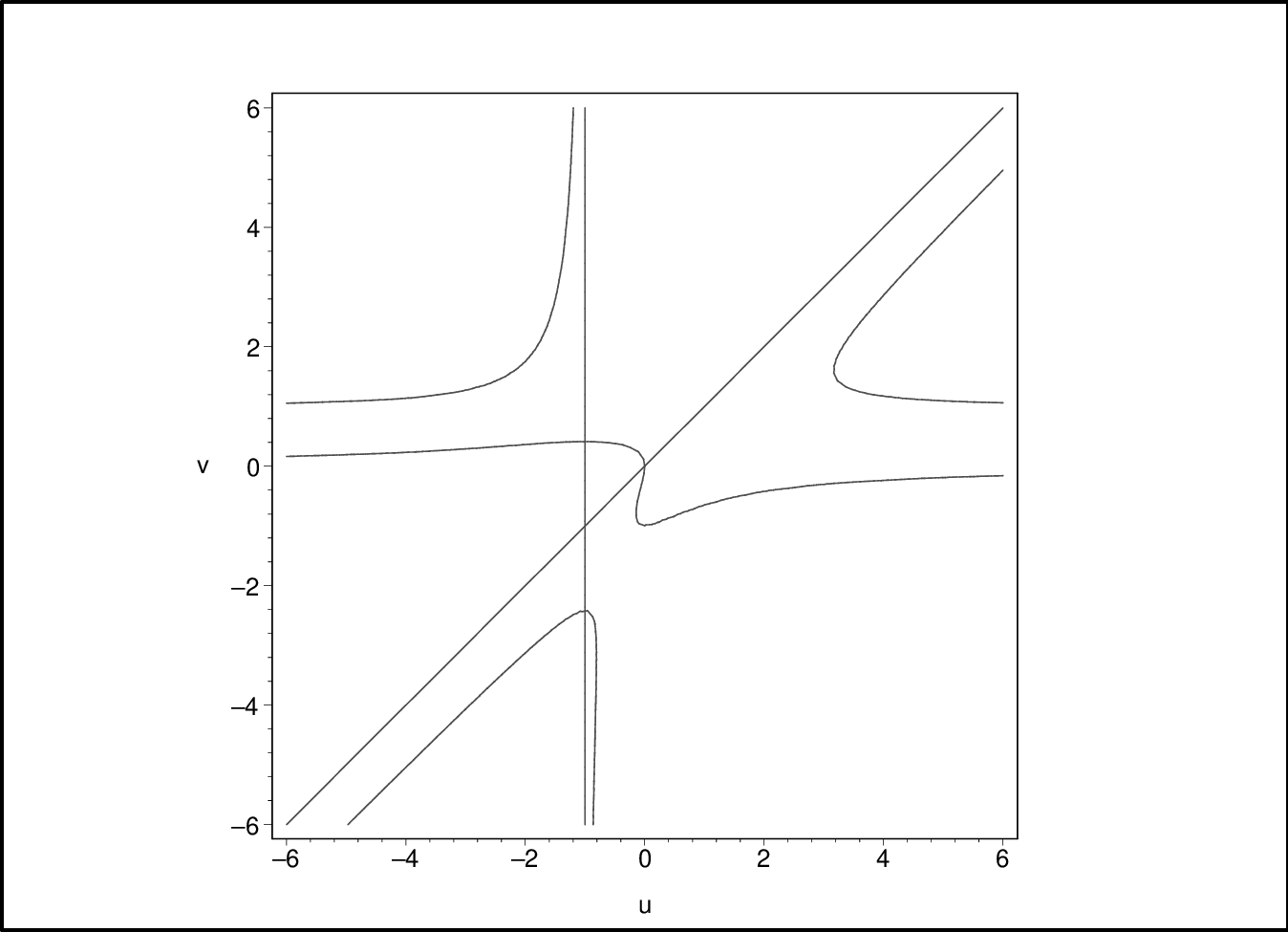}
\caption{Curves of period $1$, $2$ and $4$ orbits on the $(u,v)$-plane.}
\label{kuva9.1}
\end{center}
\end{figure}
\clearpage

\clearpage

\begin{figure}[h!]
\begin{center}
\includegraphics[width=0.6\textwidth,angle=270]{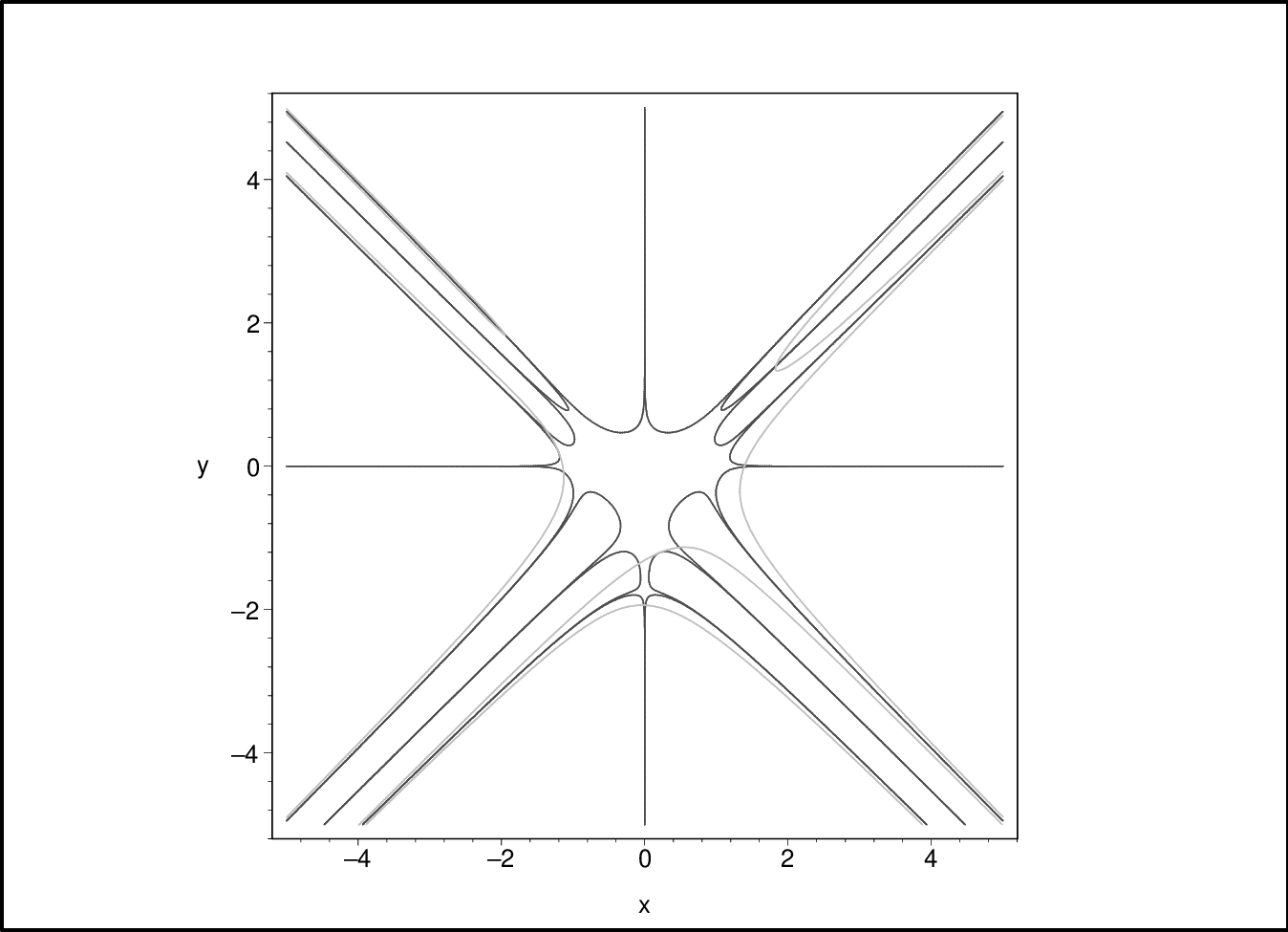}
\caption{Curves of period four orbits and eigenvalues $|\lambda_{4}(x,y)|=1$.}
\label{kuva9.3}
\end{center}
\end{figure}

\begin{figure}[h!]
\begin{center}
\includegraphics[width=0.6\textwidth,angle=270]{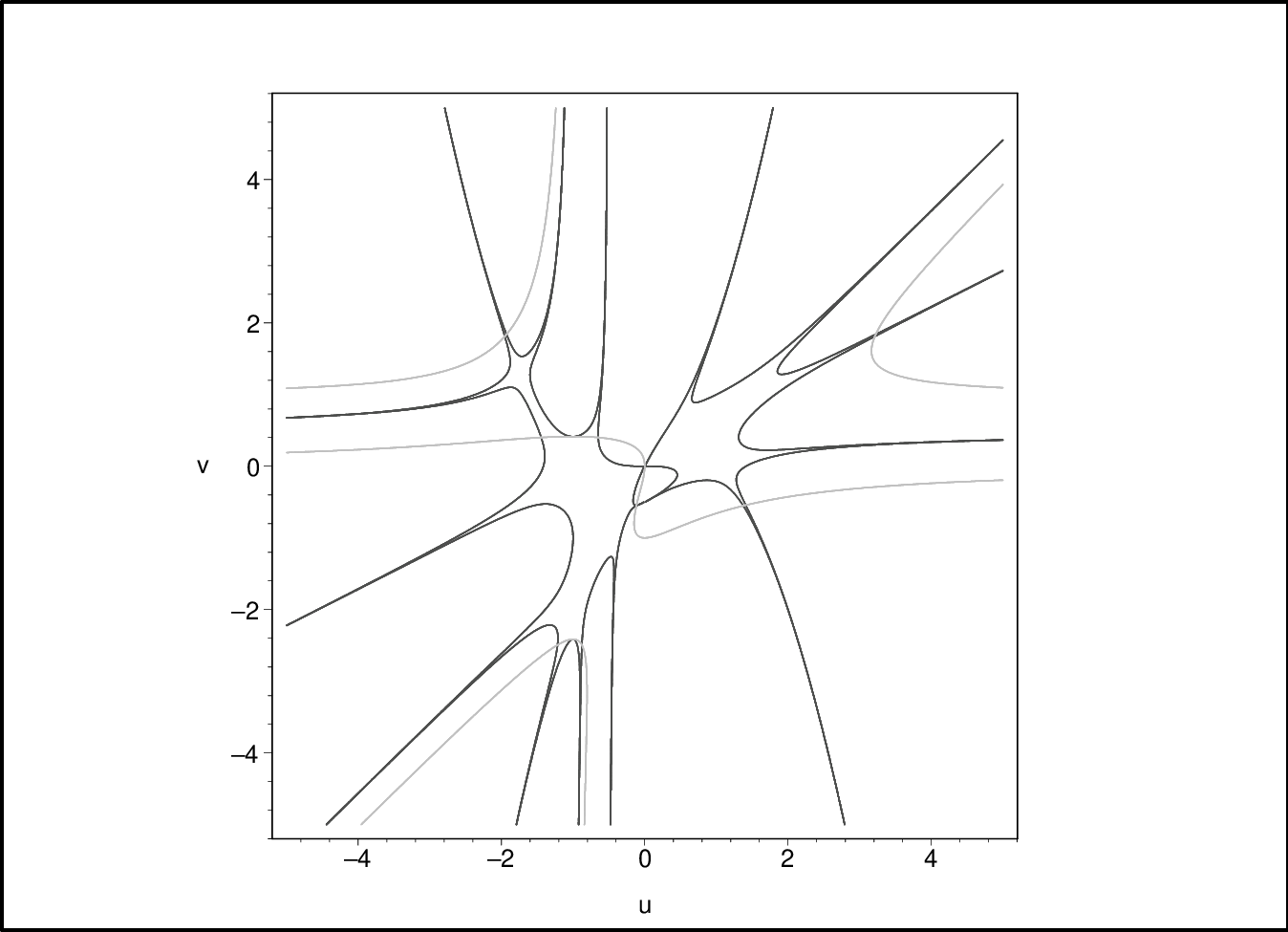}
\caption{Curves of period four orbits and eigenvalues $|\lambda_{4}(u,v)|=1$.}
\label{kuva9.4}
\end{center}
\end{figure}

\clearpage

\clearpage
\begin{figure}[h!]
\begin{center}
\includegraphics[width=0.6\textwidth,angle=270]{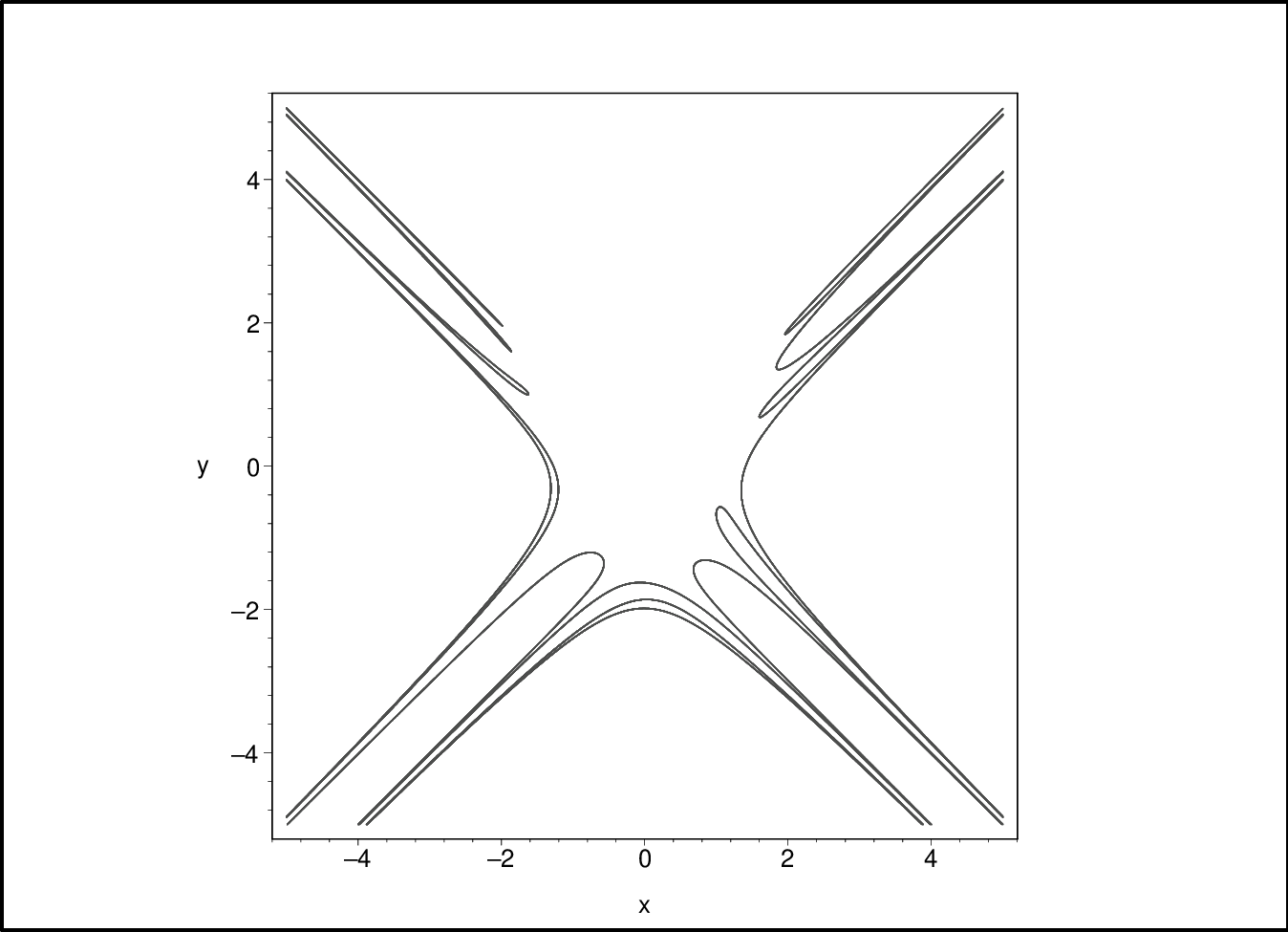}
\caption{Curves of period $5$ orbits on the $(x,y)$-plane.} \label{kuva13}
\end{center}
\end{figure}
\begin{figure}[h!]
\begin{center}
\includegraphics[width=0.6\textwidth,angle=270]{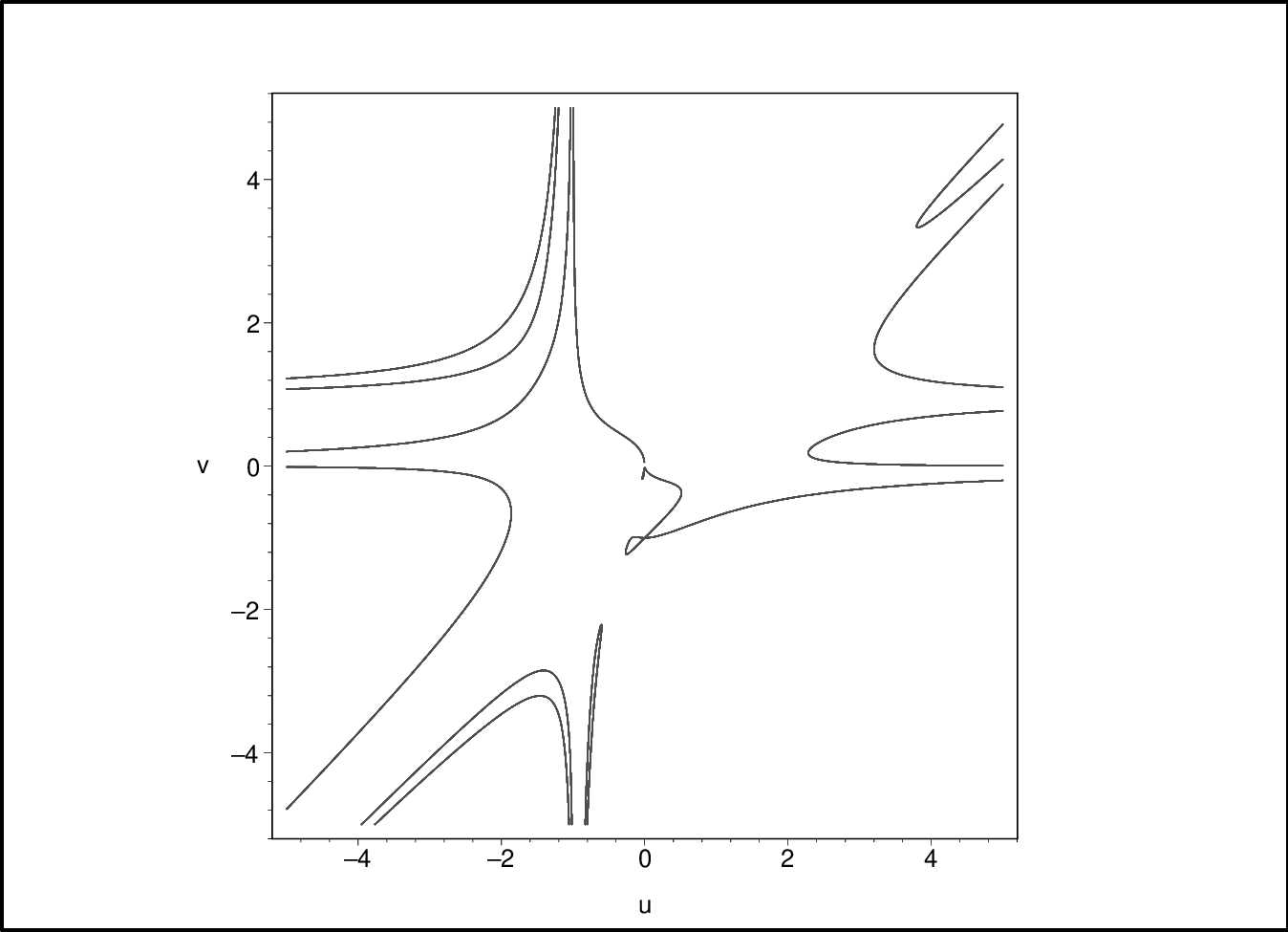}
\caption{Curves of period $5$ orbits on the $(u,v)$-plane.} \label{kuva14}
\end{center}
\end{figure}
\clearpage

\clearpage

\begin{figure}[h!]
\begin{center}
\includegraphics[width=0.6\textwidth,angle=270]{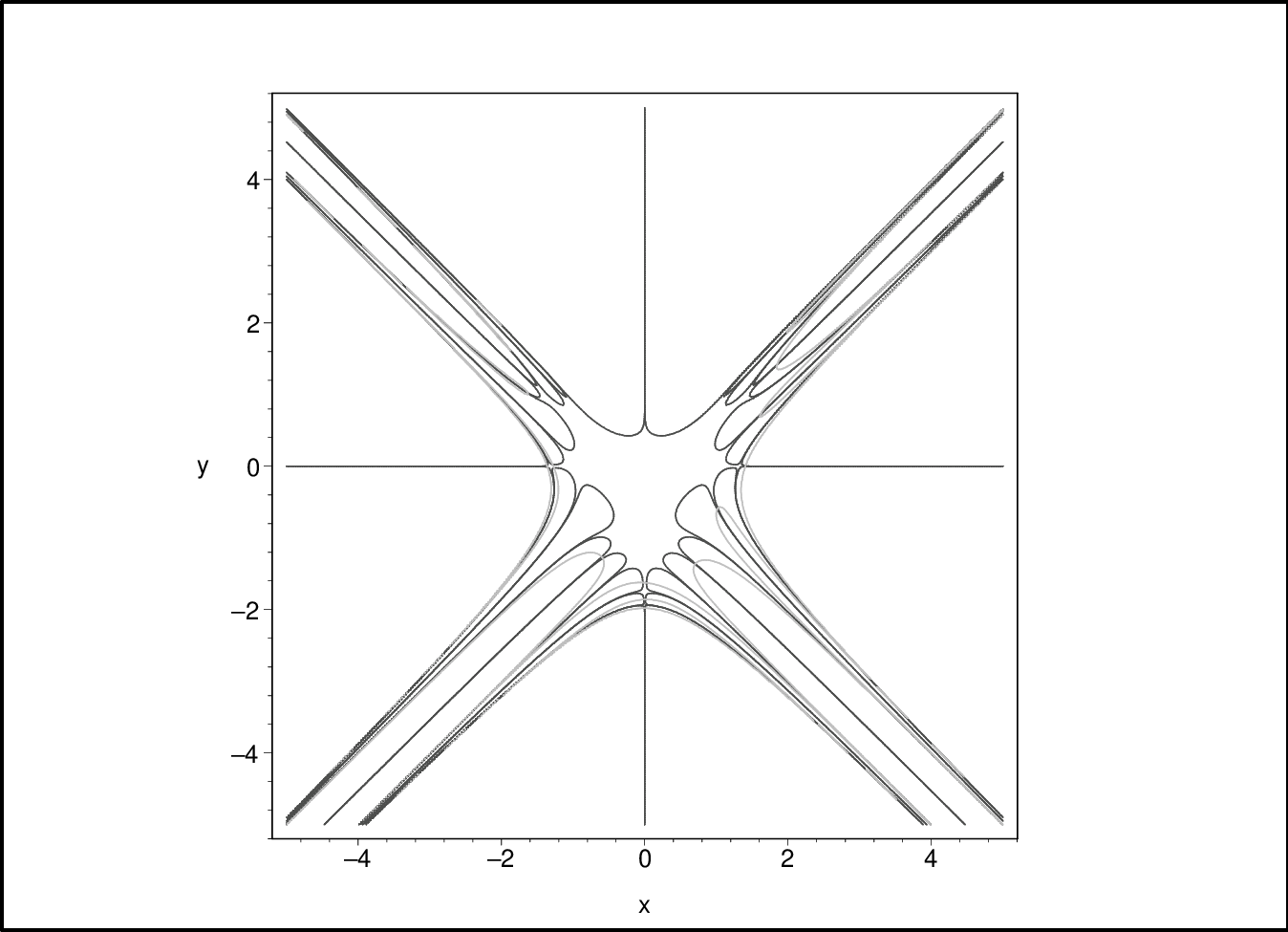}
\caption{Curves of period five orbits and eigenvalues $|\lambda_{5}(x,y)|=1$.}
\label{kuva10.11}
\end{center}
\end{figure}
\begin{figure}[h!]
\begin{center}
\includegraphics[width=0.6\textwidth,angle=270]{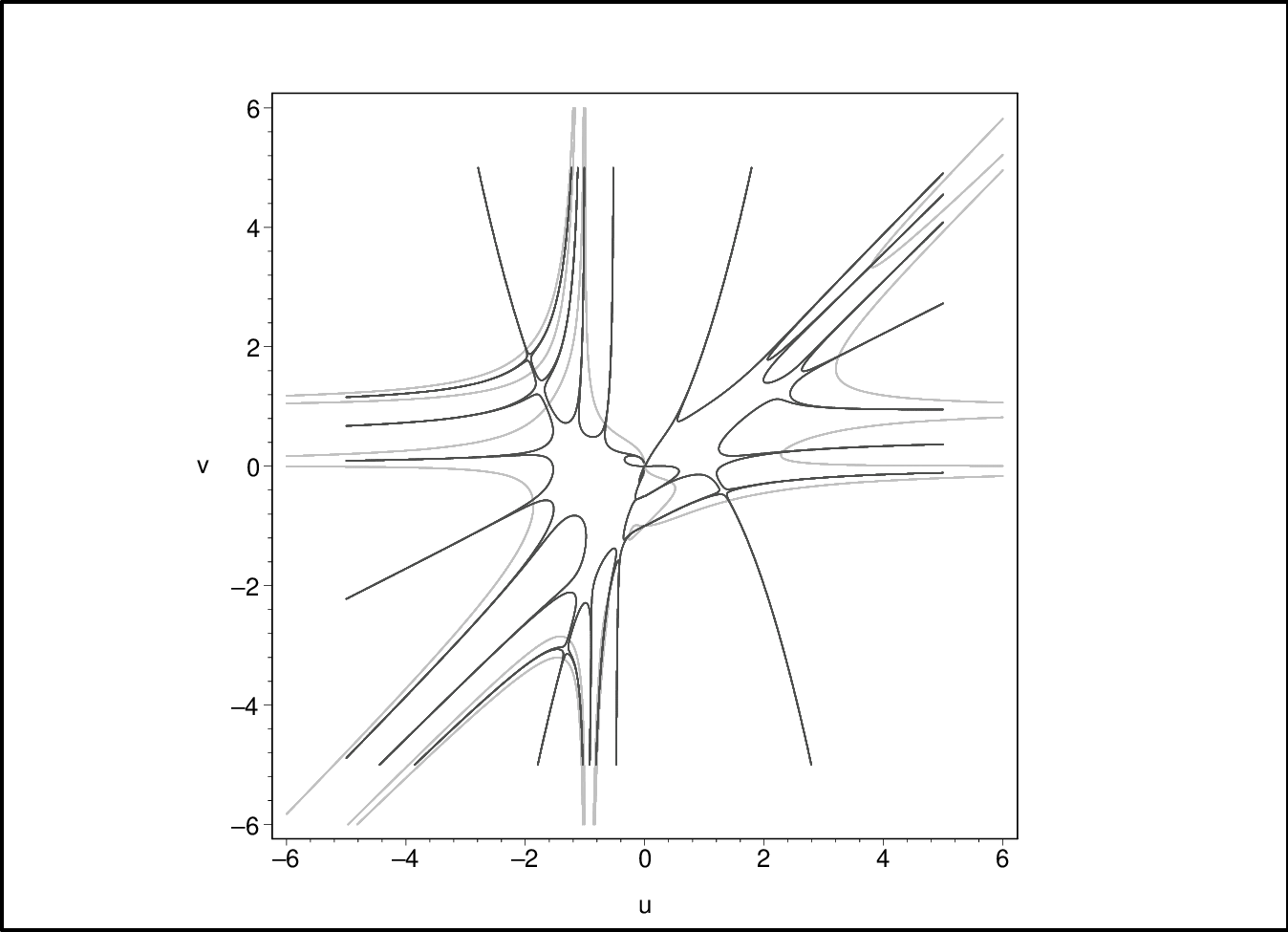}
\caption{Curves of period five orbits and eigenvalues $|\lambda_{5}(u,v)|=1$.}
\label{kuva10.1}
\end{center}
\end{figure}
\clearpage

\begin{figure}[h!]
\begin{center}
\includegraphics[width=0.6\textwidth,angle=270]{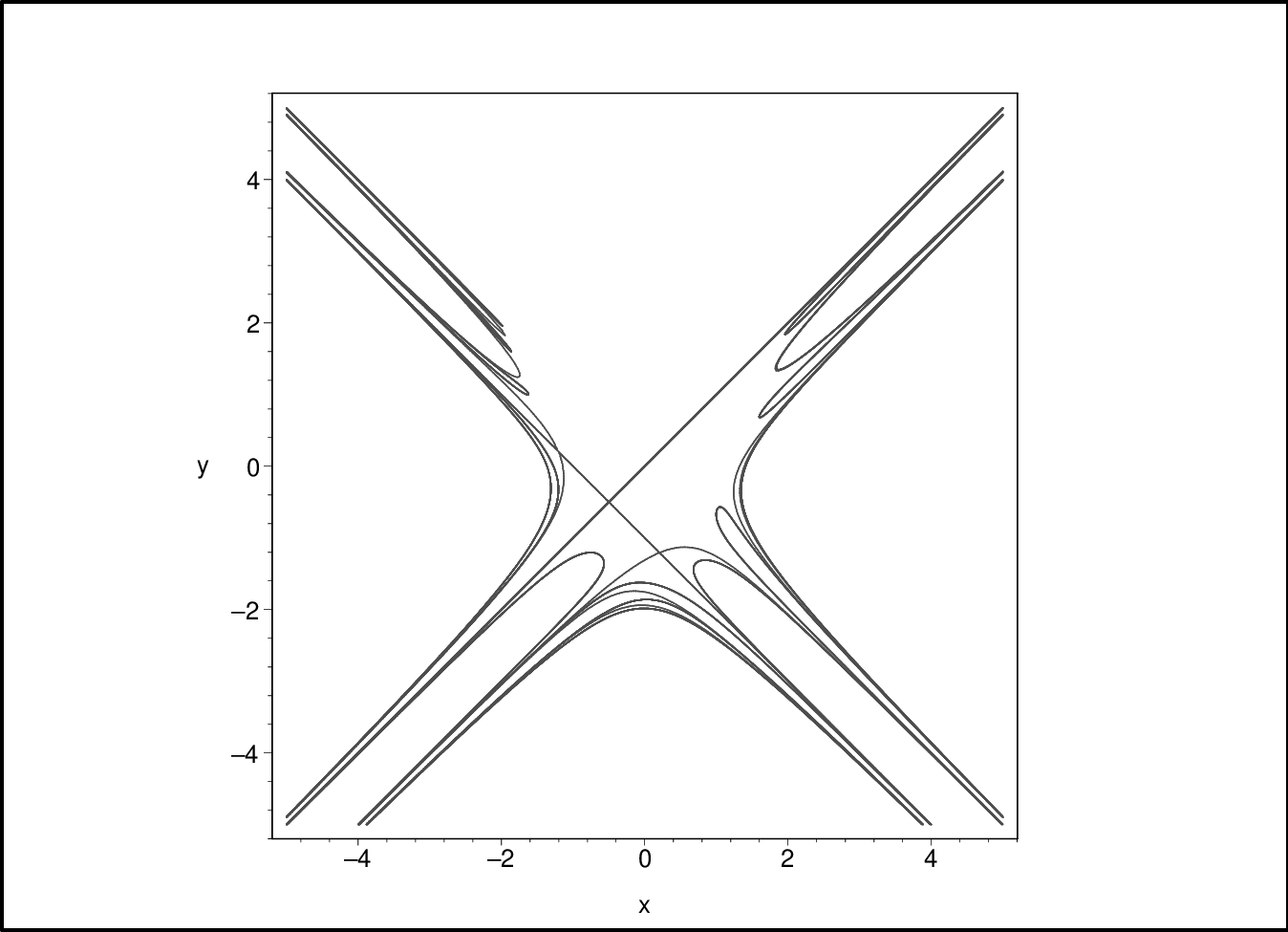}
\caption{Curves of period $1-5$ orbits on the $(x,y)$-plane.} \label{kuva16}
\end{center}
\end{figure}
\begin{figure}[h!]
\begin{center}
\includegraphics[width=0.6\textwidth,angle=270]{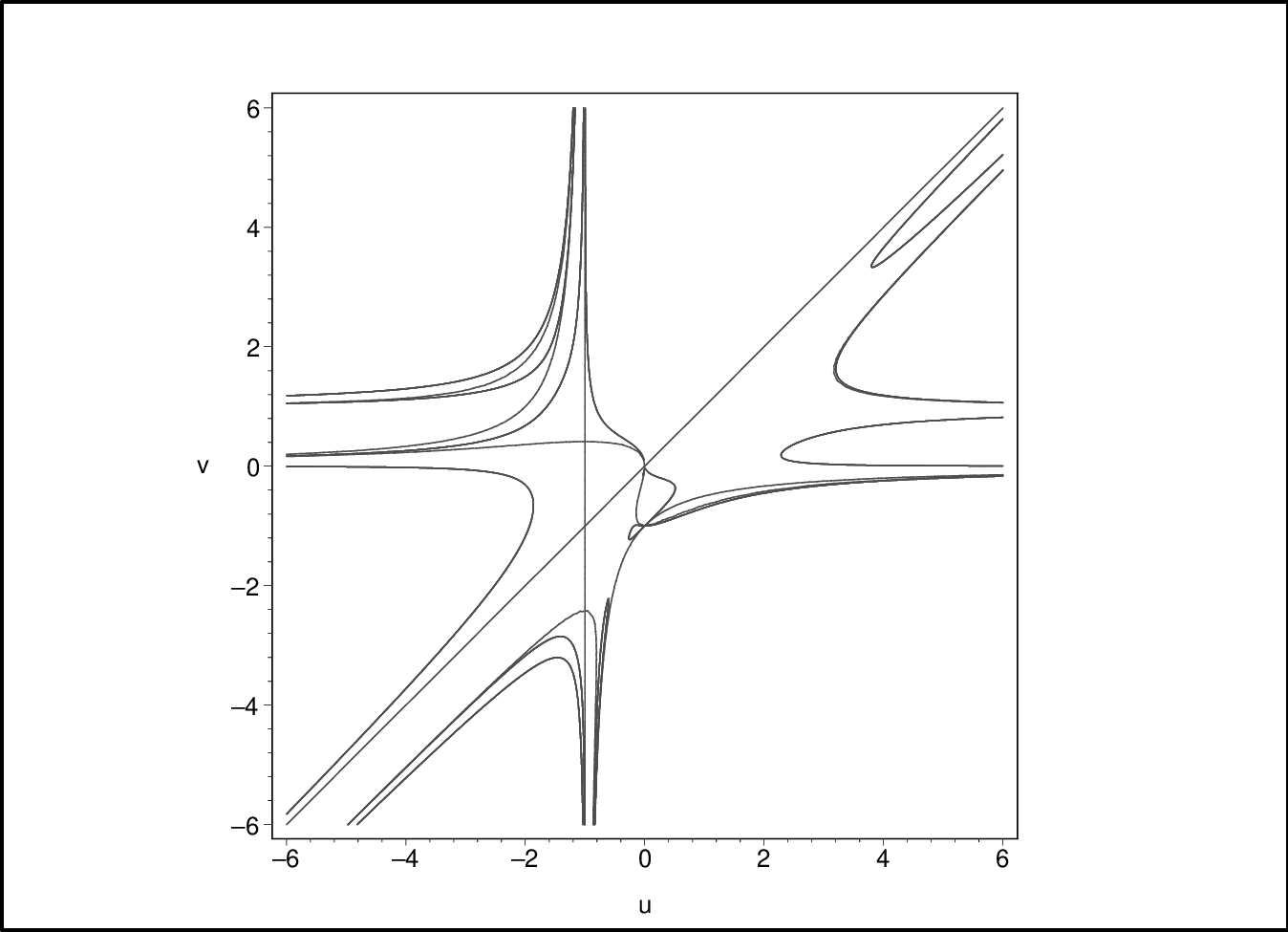}
\caption{Curves of period $1-5$ orbits on the $(u,v)$-plane.} \label{kuva17}
\end{center}
\end{figure}
\clearpage

\clearpage
\begin{figure}[h!]
\begin{center}
\includegraphics[width=0.55\textwidth]{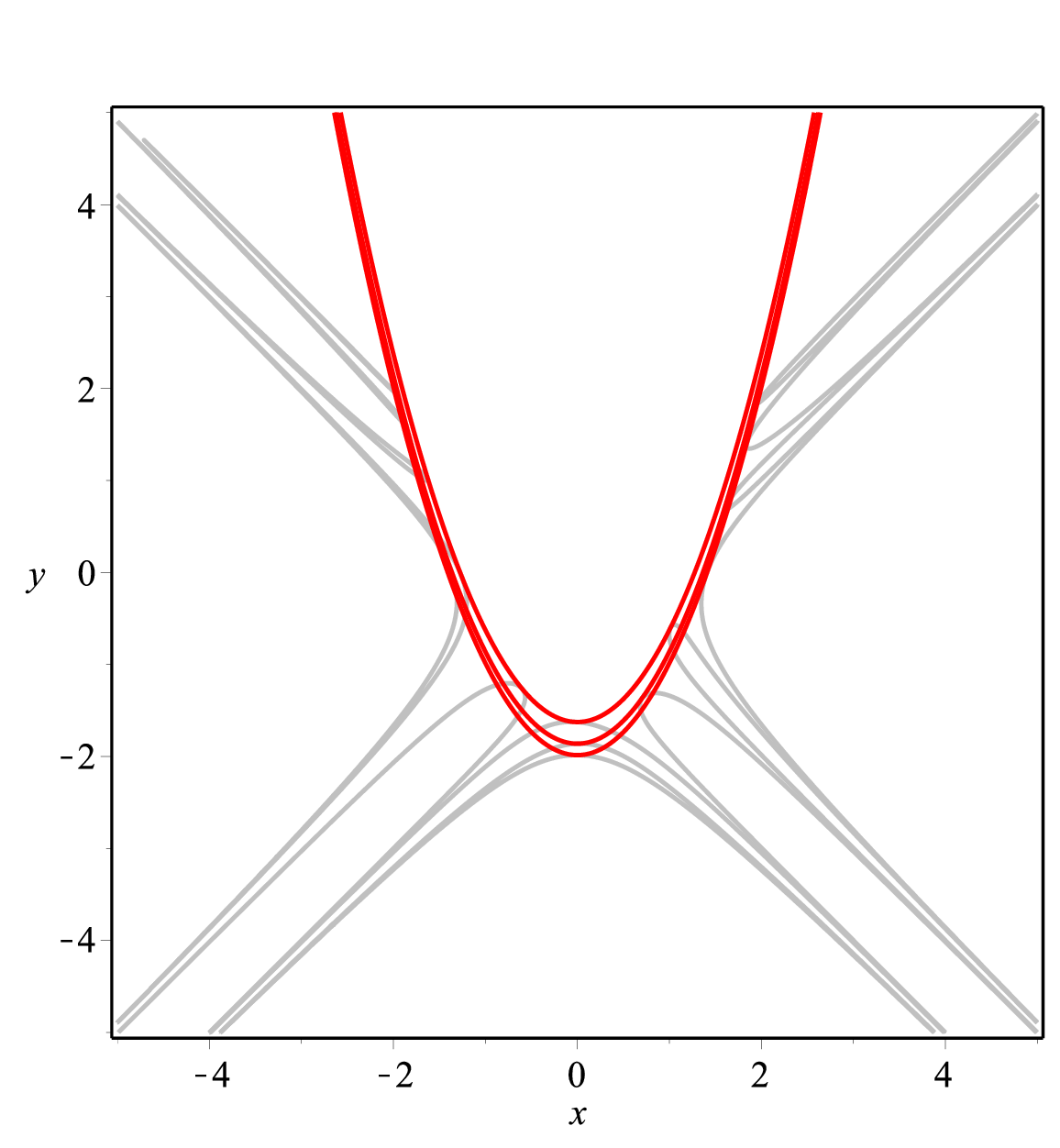}
\caption{Curves of period $5$ orbits and curves $c(x,y)=a$, when $a$ is $-1.625414$, $-1.860783$ and $-1.985424$ on the $(x,y)$-plane.} \label{kuva130}
\end{center}
\end{figure}
\begin{figure}[h!]
\begin{center}
\includegraphics[width=0.55\textwidth]{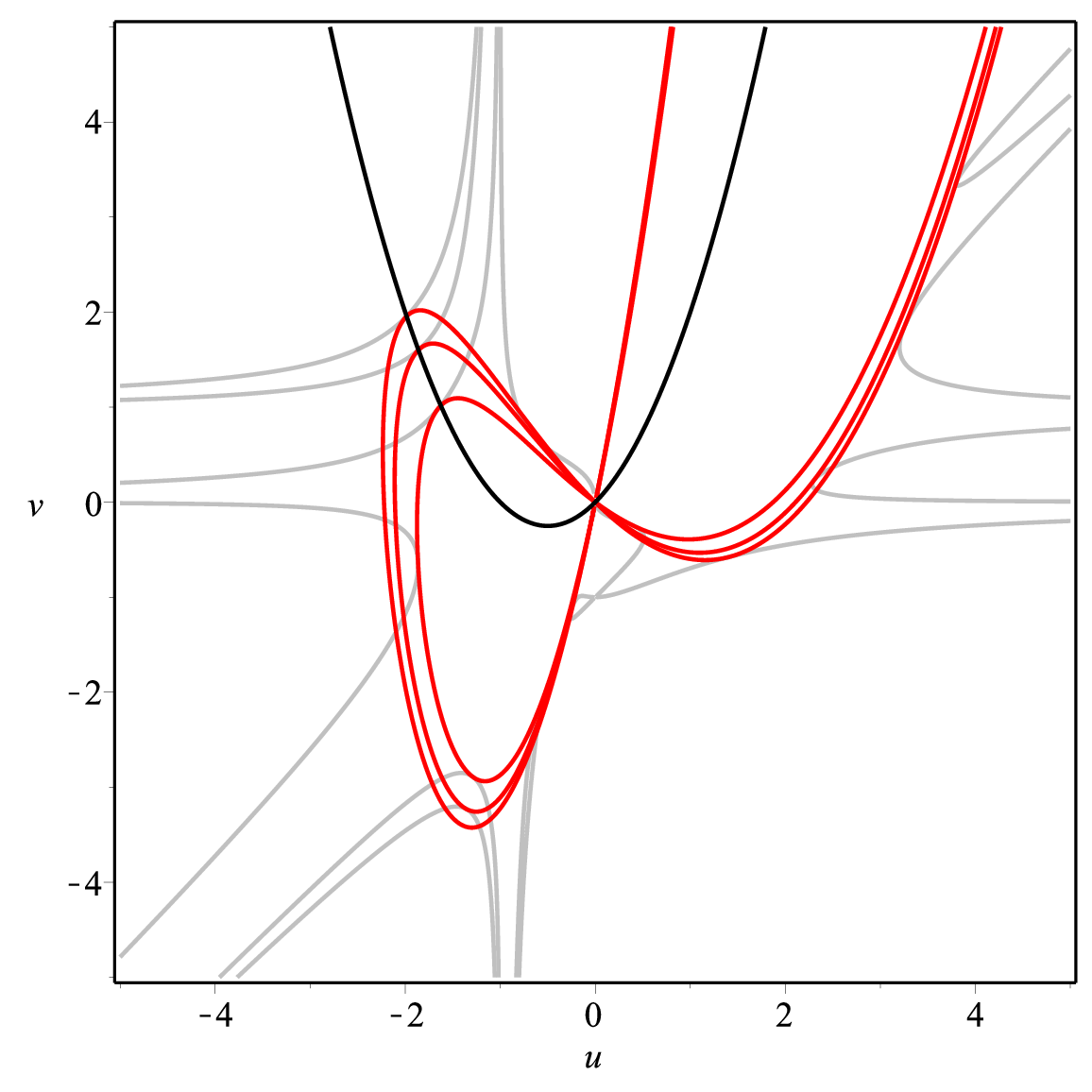}
\caption{Curves of period $5$ orbits, critical curve $v=u^2+u$ and curves $c(u,v)=a$, when $a$ is $-1.625414$, $-1.860783$ and $-1.985424$ on the $(u,v)$-plane.} \label{kuva140}
\end{center}
\end{figure}

\clearpage

\begin{figure}[h!]
\begin{center}
\includegraphics[width=0.8\textwidth]{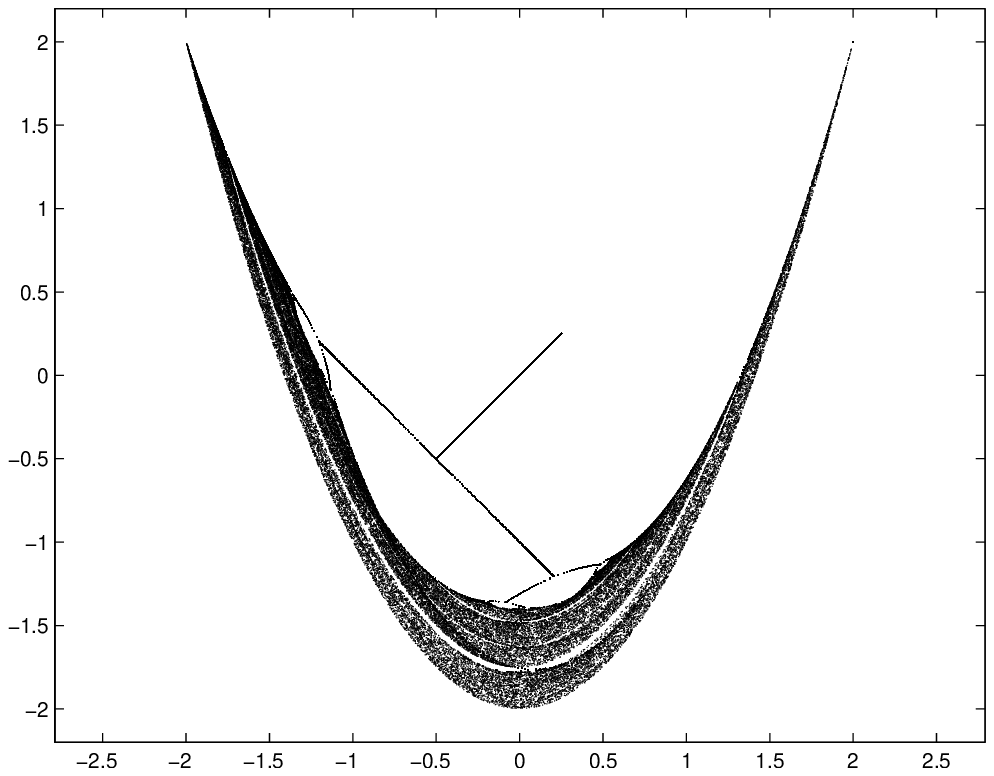}
\caption{The bifurcation diagram of the function $F(x,y)=(y,y^{2}+y-x^{2})$.}
\label{kuva18}
\end{center}
\end{figure}

\begin{figure}[h!]
\begin{center}
\includegraphics[width=0.8\textwidth]{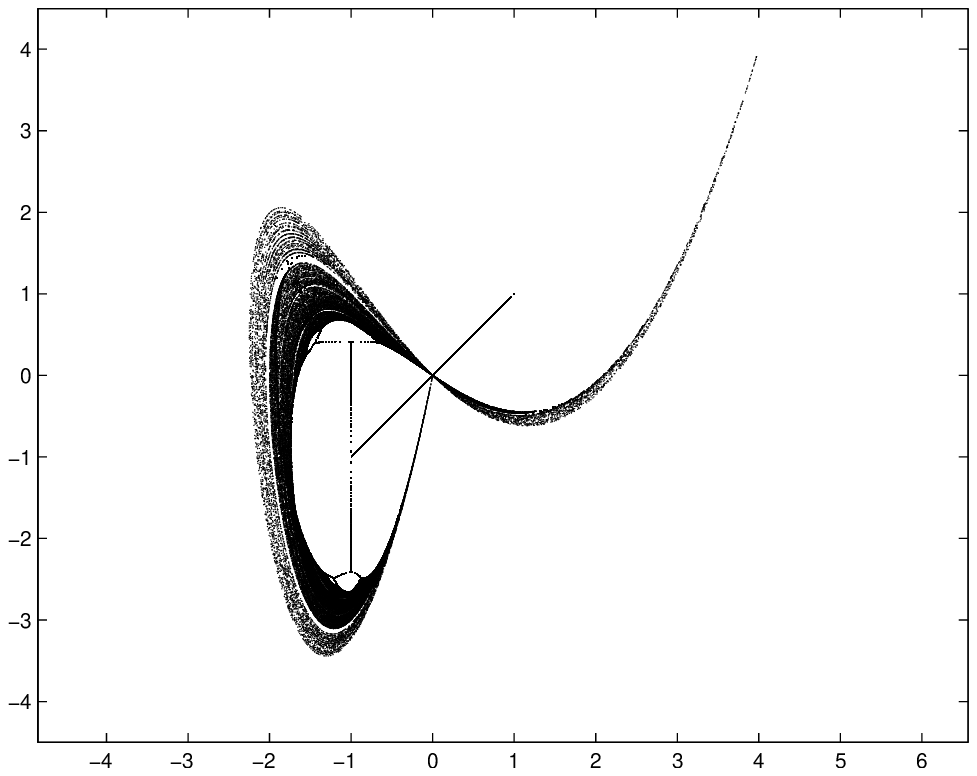}
\caption{The bifurcation diagram of the function $G(u,v)$.}
\label{kuva19}
\end{center}
\end{figure}

\end{document}